\documentclass[leqno]{article}
\usepackage{CJK,CJKnumb,CJKulem,times,dsfont,ifthen,mathrsfs,latexsym,amsfonts,color}
\usepackage{amsmath,amsthm,makeidx,fontenc,amssymb,bm,graphicx,psfrag,listings,curves,extarrows,enumitem}
\usepackage{hyperref}
\usepackage{geometry}
\usepackage{indentfirst}
\usepackage{amssymb}
\makeatletter

\newcommand{\Rmnum}[1]{\expandafter\@slowromancap\romannumeral #1@}
\makeatother
\usepackage[showonlyrefs]{mathtools}  %% cites only the labels which are quoted
\mathtoolsset{showonlyrefs=true}

\newtheorem{theorem}{Theorem}[section]
\newtheorem{lemma}{Lemma}[section]

\newtheorem{remark}{Remark}[section]

			\newcommand{\N}{{\mathbb N}}

			\newcommand{\C}{{\mathbb C}}
			\newcommand{\R}{{\mathbb R}}
			\newcommand{\bean}{\begin{eqnarray*}}
				\newcommand{\eean}{\end{eqnarray*}}

			\setitemize{itemindent=38pt,leftmargin=0pt,itemsep=-0.4ex,listparindent=26pt,partopsep=0pt,parsep=0.5ex,topsep=-0.25ex}

\numberwithin{equation}{section}
		
\begin{document}
\theoremstyle{plain}			
\title{\bf Existence and asymptotics of normalized solutions for logarithmic Schr\"{o}dinger system
\thanks{Supported by NSFC-12171265.
E-mail addresses: zhangqian9115@mail.tsinghua.edu.cn (Q. Zhang), zou-wm@mail.tsinghua.edu.cn (W. M. Zou)} }			
\date{}
\author{
{\bf Qian Zhang and Wenming Zou}\\
\footnotesize
\it  Department of Mathematical Sciences, Tsinghua University, Beijing 100084, China.\\ }							
%%%%%%%-------------------------------??===Abstract-------------------------------
\maketitle
				
\begin{center}
\begin{minipage}{120mm}
\begin{center}{\bf Abstract}\end{center}		
This paper is concerned with the following logarithmic Schr\"{o}dinger system:
$$\left\{\aligned\ &\ -\Delta u_1+\omega_1u_1=\mu_1 u_1\log u_1^2+\frac{2p}{p+q}|u_2|^{q}|u_1|^{p-2}u_1,\\
\ &\  -\Delta u_2+\omega_2u_2=\mu_2 u_2\log u_2^2+\frac{2q}{p+q}|u_1|^{p}|u_2|^{q-2}u_2,\\
\ &\ \int_{\Omega}|u_i|^2\,dx=\rho_i,\ \ i=1,2,\\
\ &\ (u_1,u_2)\in H_0^1(\Omega;\mathbb R^2),\endaligned\right.$$
where $\Omega=\R^N$ or $\Omega\subset\mathbb R^N(N\geq3)$ is a
bounded smooth domain, $\omega_i\in\mathbb R$, $\mu_i,\ \rho_i>0,\ i=1,2.$ Moreover,
$p,\ q\geq1,\ 2\leq p+q\leqslant 2^*$, where $2^*:=\frac{2N}{N-2}$. By using a Gagliardo-Nirenberg inequality and careful estimation
of $u\log u^2$, firstly, we will provide a unified proof of the existence of the
normalized ground states solution for all $2\leq p+q\leqslant 2^*$.
Secondly, we consider the stability of normalized
ground states solutions. Finally, we analyze the behavior of
solutions for Sobolev-subcritical
 case and pass the limit as the exponent $p+q$ approaches
to $2^*$.
 Notably, the uncertainty of sign of
$u\log u^2$ in $(0,+\infty)$ is one of the difficulties of
 this paper, and also one of the
motivations we are interested in. In particular, we can establish the
existence of positive normalized ground states solutions for
the Br\'{e}zis-Nirenberg type problem with
logarithmic perturbations (i.e., $p+q=2^*$). 
In addition, our study includes proving the existence of solutions to the 
logarithmic type Br\'{e}zis-Nirenberg problem with and without the $L^2$-mass 
$\int_{\Omega}|u_i|^2\,dx=\rho_i(i=1,2)$ constraint by two different methods, respectively.
Our results seems to be the first result
of the normalized solution of the coupled nonlinear Schr\"{o}dinger system with logarithmic perturbation.

\vskip0.1in					
\end{minipage}
\end{center}
%%%%%%-------------------------------Section 1=======????====Introduction-------------------------------
\vskip0.18in
\section{Introduction}
Consider the following coupled Schr\"{o}dinger system with   logarithmic terms:
\begin{equation}\label{eq:1}
	\left\{\aligned\ &\ -\Delta u_1+\omega_1u_1=\mu_1 u_1\log u_1^2+\frac{2p}{p+q}|u_2|^{q}|u_1|^{p-2}u_1,\\
\ &\  -\Delta u_2+\omega_2u_2=\mu_2 u_2\log u_2^2+\frac{2q}{p+q}|u_1|^{p}|u_2|^{q-2}u_2,\\
\ &\ \int_{\Omega}|u_i|^2\,dx=\rho_i,\ \ i=1,2,\\
\ &\ (u_1,u_2)\in H_0^1(\Omega;\mathbb R^2),\endaligned\right.
\end{equation}
where $\Omega=\R^N$ or $\Omega\subset\mathbb R^N(N\geq3)$ is a
bounded smooth domain,
$\omega_i\in\mathbb R$, $\mu_i,\ \rho_i>0,\ i=1,2$ and
$p,\ q\geq 1,\ 2\leq p+q\leqslant 2^*$, where $2^*:=\frac{2N}{N-2}$.
System \eqref{eq:1} is closely related to the time-dependent nonlinear
logarithmic type Schr\"{o}dinger system:
\begin{equation}\label{eq:system_schro}
\left\{\aligned\ &\ \text{i}\partial_t\Psi_1+\Delta\Psi_1
+\Psi_1(\mu_1\log\Psi_1^2+|\Psi_2|^{q}|\Psi_1|^{p-2})=0,\\
\ &\ \text{i}\partial_t\Psi_2+\Delta\Psi_2+\Phi_2
(\mu_2\log\Psi_2^2+|\Psi_1|^{p}|\Psi_2|^{q-2})=0,\endaligned\right.
\end{equation}
where $\text{i}$ denotes the imaginary unit, $\Psi_i:(0,+\infty)\times\Omega\to\mathbb C$
and, for every $t>0$, $\Psi_i(t,\cdot)\in H^1_0(\Omega;\mathbb C),\ i=1,2$.
Problem \eqref{eq:system_schro} appears
in various physical fields, for instance in quantum mechanics,
quantum optics, nuclear physics, transport and diffusion phenomena, open quantum
systems, effective quantum gravity, theory of superfluidity and Bose-Einstein condensation. We refer the
readers to \cite{1,4,3,8,7,001,29,27} and their references for more
information on the related physical backgrounds.

\vskip0.1in
The solution  to \eqref{eq:system_schro} satisfies
the following conserved quantities:
$$\mathcal{Q}(\Psi_i)=\int_\Omega|\Psi_i|^2\,dx,\ \ i=1,2,$$
$$\aligned\mathcal{E}(\Psi_1,\Psi_2):&=\frac{1}{2}\int_{\Omega}\left(|\nabla\Psi_1|^2
+|\nabla\Psi_2|^2-\mu_1\Psi_1^2(\log\Psi_1^2-1)
-\frac{4}{p+q}|\Psi_1|^{p}|\Psi_2|^{q}-\mu_2\Psi_2^2(\log\Psi_2^2-1)\right)\,dx.\endaligned$$
In this paper, we will look for standing wave solutions
$(\Psi_1(t,x),\Psi_2(t,x))=(e^{i\omega_1 t}u_1(x),e^{i\omega_2 t}u_2(x))$ of
\eqref{eq:system_schro}
such that $(u_1,u_2)\in H^1_0(\Omega;\mathbb R^2)$ and
\begin{equation}\label{eq:mass_constraint}
\mathcal{Q}(u_i)=\rho_i,\ \ i=1,2,
\end{equation}
for some fixed $\rho_1,\rho_2>0$.

We recall that the study of the a single equation with logarithmic
perturbation:
\begin{equation}\label{eq3}
-\Delta u+V(x)u=u\log u^2,\ \ \ \ x\in\R^N.
\end{equation}
The logarithmic term  $u\,\text{log}\,u^2$ brings about essential difficulties
for getting the solutions of \eqref{eq:1}. In particular, the functional associated with the logarithmic term
$$\mathcal{K}(u):=\int_{\R^N}u^2\log u^2\,dx$$
is not differentiable in the space $H^1(\R^N)$. According to the following
standard logarithmic
Sobolev inequality (see \cite[Theorem 8.14]{LiebLoss}):
\begin{equation}\label{eq:5}
\int_{\R^N}u^2\log u^2\,dx\leq \frac{a^2}{\pi}|\nabla u|_{L^2(\R^N)}^2+(\log |u|_{L^2(\R^N)}^2
-N(1+\log a))|u|_{L^2(\R^N)}^2,
\end{equation}
where $u\in H^1(\R^N)$, $a>0$. Notice that $\mathcal{K}(u)<+\infty$ for all
$u\in H^1(\R^N)$, but there exists
$u\in H^1(\R^N)$ such that $\mathcal{K}(u)=-\infty$. Indeed, choosing smooth function
$$u(x)=\left\{\aligned\ &\ (|x|^{\frac N2}\log|x|)^{-1},\ \ &|x|\geq 3,\\
\ &\ 0,\ \ &|x|\leq 2,\endaligned\right.$$
one can verify directly that $u\in H^1(\R^N)$, but $\mathcal{K}(u)=-\infty$. Thus, in
general,
$\int_{\R^N} u^2\log u^2\,dx$ fails to be finite and losses $C^1$-smoothness on
$H^1(\R^N)$.
In order to overcome this technical difficulty and study the existence of solutions, several approaches developed so far
in the literature. To make the corresponding
functional of problem \eqref{eq3} well defined and $C^1$
smooth, Cazenave \cite{mma9} worked in an Orlicz space. By applying non-smooth critical point theory for lower semi-continuous
functionals, Squassina and Szulkin \cite{mma23} studied the existence of the positive ground state solution
to the logarithmic Schr\"{o}dinger equation \eqref{eq3} with periodic
potential. Via the directional derivative and constrained minimization method, Shuai \cite{mma22} investigated
the existence of positive ground state solution for the logarithmic
Schr\"{o}dinger equation \eqref{eq3} under different types of potentials. Also using direction
derivative and constrained minimization method, Deng, Pi and Shuai \cite{mma} considered the existence
of the positive ground state solutions  to  the logarithmic Schr\"{o}dinger equation \eqref{eq3}
in Sobolev-subcritical and Sobolev-critical cases. However, the normalized solutions with $L^2$ constraints are not considered there.

%In addition, they analyze the behavior of solutions
%for subcritical problem and pass the limit as the exponent
%$p$ approaches to $2^*$. A natural and interesting question is what happens to the
%behavior of the solution as $p+q$ tends to $2^*$ for the coupled system \eqref{eq:1}?

\vskip0.1in
Recently, Deng, He, Pan and Zhong
\cite{deng} proved the existence of positive solution for the
equation
\begin{equation}\label{eq4}
	-\Delta u+\omega u=\mu u\log u^2+|u|^{p-2}u\ \ \text{in}\ \ \R^N,
	\end{equation}
where $\Omega\subset\R^N(N\geq 3)$ is a bounded smooth domain, $\lambda,\mu\in\R$.
The equation \eqref{eq4} considered in the papers \cite{deng} and \cite{mma} are closely 
related to the classical Br\'{e}zis-Nirenberg type problem:
\begin{equation}\label{eq7}
-\Delta u=|u|^{2^*-2}u+\lambda u\ \ \text{in}\ \ \Omega\subset
 {\mathbb{R}}^{N}(N\geq 3).
\end{equation}
Br\'{e}zis and Nirenberg \cite{5}  found out that the existence of a solution depends heavily
on the values of $\lambda$ and $N$.

\vskip0.1in

To the best of our knowledge, it seems that there is no relevant result about the  logarithmic
Schr\"{o}dinger system \eqref{eq:1}. So, the second purpose of this paper is to fill this gap for the  logarithmic
Schr\"{o}dinger system.

\vskip0.12in

The motivation we consider is the normalized solutions of system \eqref{eq:1}.  When the nonlinear Schr\"{o}dinger equations or systems  have  pure power nonlinearities and are defined
on the whole space $\R^N$, the existence of the normalized solutions
 have been widely studied. We only refer the readers to
\cite{n5,n7,n9,n23,n25} and  more recent recent papers
\cite{mzhong,mlu,mmolle} and the references therein. However, when we consider such problems defined on  the  bounded smooth domains of $\R^N(N\geqslant1)$, the situation will be rather different.
A  key difference is that $\mathbb R^N$ is invariant under
translations and dilations. As we know, translations are
responsible for a loss of compactness. Further, dilations can be used to
produce variations and eventually construct natural constraints such as the so-called
Poho\v{z}aev manifold.  Note that  Noris, Tavares
and Verzini \cite{APDE} considered the following equation with
power-type nonlinearities:
\begin{equation}\label{eq2}
\left\{\aligned&-\Delta u+\lambda u=|u|^{p-1}u\ \ \text{in}\ \ \Omega,\\
&\int_{\Omega}u^2\,dx=\rho>0,\endaligned\right.
\end{equation}
where both the existence and orbital stability on the spherical domain $\Omega=B_1$ are studied. After that,
in \cite{CVPDE}, Pierotti and Verzini considered the general bounded domain. For the nonlinear Schr\"{o}dinger
system, Noris, Tavares and Verzini \cite{Nonlinearity} considered the existence and stability of the normalized
solutions on a bounded domain.

Since there is no result in this direction for
logarithmic Schr\"{o}dinger system, our third motivation is to prove the existence and
stability of the normalized solutions of \eqref{eq:1} on a general bounded domain
$\Omega\subset\R^N$.

We now formulate our problem in a variational setting as follows: consider
$$\aligned\mathcal{I}(u_1,u_2):&=
\frac{1}{2}\int_{\Omega}(|\nabla u_1|^2+|\nabla u_2|^2+\omega_1 u_1+\omega _2 u_2)\,dx\\
&\ \ \ \ -\frac{1}{2}\int_\Omega\left(\mu_1 u_1^2(\log u_1^2-1)
+\frac{4}{p+q}|u_1|^{p}|u_2|^{q}+
\mu_2 u_2^2(\log u_2^2-1)\right)\,dx,
\endaligned$$
which is well defined on the space $\mathcal{H}(\Omega):=H_1\times H_2$, where
$$H_i:=\{u_i\in H_0^1(\Omega): u_i^2\log u_i^2\in L^1(\Omega)\},\ i=1,2.$$
When $u_i=0$, we define $u_i^2\log u_i^2=0$.
If $\Omega\subset\R^N$ be a smooth bounded domain, $(u_1,u_2)\in H_0^1(\Omega;\R^2)$ and 
for all $\varphi_1,\ \varphi_2\in C_0^\infty(\Omega)$, we
compute the Gateaux derivative
\begin{equation}\label{eq37-5}
\aligned \langle\mathcal{I}^\prime (u_1,u_2),(\varphi_1,\varphi_2)\rangle:&=
\int_{\Omega} (\nabla u_1\nabla \varphi _1+\omega_1u_1\varphi_1+
\nabla u_2\nabla \varphi _2+\omega_2u_2\varphi _2 )\,dx\\
&\ \ \ \
-\int_\Omega\bigg(\mu_1 u_1\varphi _1\log u_1^2
+\mu_2 u_2\varphi _2\log u_2^2
+\frac{2p}{p+q}|u_2|^{q}|u_1|^{p-2}u_1\varphi_1\\
&\ \ \ \
+\frac{2q}{p+q}|u_1|^{p}|u_2|^{q-2}u_2\varphi_2\bigg)\,dx.\endaligned
\end{equation}
Moreover, $(u_1, u_2)$ is a normalized solution of system \eqref{eq:1}, that is, solutions of \eqref{eq:1} can be
 identified with critical points of the associated energy functional
$$\aligned\mathcal{E}(u_1,u_2)&=\frac{1}{2}\int_{\Omega}(|\nabla u_1|^2+|\nabla u_2|^2)\,dx\\
&\ \ \ \ -\frac{1}{2}\int_\Omega\left(\mu_1 u_1^2(\log u_1^2-1)
+\frac{4}{p+q}|u_1|^{p}|u_2|^{q}+
\mu_2 u_2^2(\log u_2^2-1)\right)\,dx\endaligned$$
restricted to the mass constraint
\begin{equation}\label{eq:defM}
\mathcal{M}=\mathcal{M}_{\rho_1,\rho_2} :=\left\{(u_1,u_2)\in H^1_0(\Omega;\mathbb R^2):
\int_\Omega u_i^2\,dx=\rho_i,\ \ i=1,2 \right\},
\end{equation}
with $\omega_i$ playing the role of Lagrange multipliers. Our main aim is to
provide conditions on $p$ and $(\rho_1,\rho_2)$ (and also on $\mu_1,\mu_2$) so that
$\mathcal{E}|_{\mathcal{M}}$ admits local minima. We
call such solutions normalized ground states solutions. A common key tool in the study of 
normalized solution is the Gagliardo-Nirenberg inequality (see \eqref{1.12} below). 
We introduce a sharp
Gagliardo-Nirenberg inequality \cite[pp 458-462]{1},  for any
$u\in H^1(\mathbb R^N),$
\begin{equation}\label{1.12}
\int_{\R^N}|u|^{p+q}\,dx\leq C_{N,p+q}\left(\int_{\R^N}u^2\,dx\right)
^{\frac{p+q}{2}-\frac{N(p+q-2)}{4}}
\left(\int_{\R^N}|\nabla u|^2\,dx\right)^{\frac{N(p+q-2)}{4}}.
\end{equation}
It is proved in \cite{470} that
$$C_{N,p+q}=\inf_{u\in H^1(\R^N)\setminus\{0\}}\frac{\|u\|_{L^{p+q}(\R^N)}^{p+q }}
{\|\nabla u\|_{L^2(\R^N)}^{2a}\|u\|_{L^2(\R^N)}^{4r}}
=\frac{\|Z\|_{L^{p+q}(\R^N)}^{p+q }}
{\|\nabla Z\|_{L^2(\R^N)}^{2a}\|Z\|_{L^2(\R^N)}^{4r}},$$
where $Z$ is, up to translations, the unique positive solution of the equation:
\begin{equation}\label{1.10}
-\Delta Z+Z=Z^{p+q-1}.
\end{equation}
Furthermore, the exponents $a$ and $r$ as
\begin{equation}\label{eq:newexponent}
a=a(N,p,q):=\frac{N(p+q-2)}{4},\ \ \ \
r=r(N,p,q):=\frac{ p+q }{4}-\frac{N(p+q-2)}{8}.
\end{equation}
We remark that this inequality holds also in $H_0^1(\Omega)$,
for any bounded domain $\Omega$, with the
same constant $C_{N,p+q}.$ In particular, the inequality on $H_0^1(\Omega)$ is strict unless $u$ is trivial.

In the special case $p+q=2^*$ and $N\geq 3$,
\begin{equation}\label{eq:Sobolev_constant}
S_N:=C_{N,2^*}.
\end{equation}
Observe that $S_N$ is just the best Sobolev constant of the embedding
$\mathcal{D}^{1,2}(\mathbb R^N)
\hookrightarrow L^{2^*}(\mathbb R^N)$:
$$|v|_{2^*}^{2^*}\leq S_N|\nabla v|_{2}^{2^*}\ \
\text{for every}\ \ v\in
\mathcal{D}^{1,2}(\mathbb R^N).$$	

In general,  the aim of the present paper is threefold: firstly, since the existence of the  normalized solution to the logarithmic Schr\"{o}dinger system is unknown, we give
the existence and stability of the normalized solutions to the logarithmic Schr\"{o}dinger
system \eqref{eq:1} on bounded domains. Our results are even for the case of  the scalar equation.
Secondly, our study includes the Br\'{e}zis-Nirenberg type problem with logarithmic perturbations
(i.e., $p+q=2^*$). In this case, we need to overcome
two major challenges: on bounded domains, since the transformation 
introduced in \cite{n25} is not available in this case, one can not obtain a bounded Palais-Smale 
sequence and analyze its compactness as in \cite{n25} and references therein; when
faced with the Sobolev-critical case, the compactness of the Sobolev
embedding is not guaranteed. Finally, we consider the asymptotic behavior of
the solutions to the Sobolev-subcritical
case as  $p+q\to 2^*$.

\vskip0.12in

Before state  our main results, let us briefly describe the ideas of the proofs of
 Theorems \ref{prop:compact_intro}-\ref{prop:supercritical} below,
and incidentally give some notations that are needed to describe the theorems.
 When
$2\leq p+q\le 2^*$,
$\mathcal{E}|_{\mathcal{M}}$ is not
bounded below (see Lemma \ref{lemma:geometry_mp}). Nonetheless we will
show that, even though no global minima can exist, local ones do, in case
$\rho$ belongs to some explicit set. To detect the existence of such minima
we need to introduce some auxiliary problem and further notations. We follow the method of Noris, Tavares, Verzini
\cite{Nonlinearity},
let, for $\alpha\geq \lambda_1(\Omega)$, in what follows, we take
 $\Omega\subset\mathbb R^N(N\geq3)$ a Lipschitz bounded domain. Define
\begin{equation}\label{eq:Balpha0}
\mathcal{B}_\alpha:=\left\{(u_1,u_2)\in\mathcal{M}:
\int_\Omega(|\nabla u_1|^2+|\nabla u_2|^2)\,dx\leq(\rho_1+\rho_2)\alpha\right\},
\end{equation}
\begin{equation}\label{eq:Balpha}
\mathcal{U}_\alpha:=\left\{(u_1,u_2)\in\mathcal{M}:
\int_\Omega(|\nabla u_1|^2+|\nabla u_2|^2)\,dx=(\rho_1+\rho_2)\alpha\right\}.
\end{equation}
	
Since $\mathcal{B}_\alpha$ contains a pair of suitably normalized first eigenfunctions, i.e.,
$(\sqrt{\rho}_1\varphi_1,\sqrt{\rho}_2\varphi_2)$, we obtain that $\mathcal{B}_\alpha\neq \varnothing$. Moreover, let
\begin{equation}\label{eq:calpha}
c_\alpha:=\inf_{\mathcal{B}_\alpha}\mathcal{E},\ \ \  \ \
\hat{c}_\alpha:=\inf_{\mathcal{U}_\alpha}\mathcal{E}.
\end{equation}
Because $\mathcal{B}_\alpha$ weakly closed in $\mathcal{M}$, so in the
 Sobolev subcritical case, on $\mathcal{U}_\alpha\subset\mathcal{B}_\alpha$, for any $\alpha\geq \lambda_1(\Omega)$
 can reach $c_\alpha$.
Therefore, in order to find a solution of \eqref{eq:1}, it is sufficient to find
$\alpha$ such that $c_\alpha<\hat{c}_\alpha$. Furthermore, in the Sobolev-critical case, since $H^1_0(\Omega)$ is not compactly embedded in $L^{2^*}(\Omega)$,
$\mathcal{E}|_{\mathcal{M}}$ is no longer weakly lower semicontinuous. To overcome this difficulty, in
the spirit of Br\'{e}zis, Nirenberg \cite{{BL83}}, we are able to recover the compactness of
the minimizing sequences associated to
$c_\alpha$ by imposing a bound on the masses $\rho$ and on $\alpha$. More precisely, we have the
following key result.
	
\begin{theorem}[Sobolev-critical cases: $c_\alpha$ is achieved]\label{prop:compact_intro} 
	Let $\Omega\subset \R^N$  be a smooth bounded domain, consider $N\ge3$ and $p+q=2^*$. Suppose that
$\alpha\geq\lambda_1(\Omega)$ are such that
\begin{equation}\label{eq:compact_intro}
(\rho_1+\rho_2)\left(\alpha-\lambda_1(\Omega)\right)\le\frac{1}{\Lambda^{\frac{N-2}{2} }},
\end{equation}
where
\begin{equation}\label{eq:def_Lambda}
\Lambda:=\max_{\{x^2+y^2=1\}}S(N)\left(\frac{(N-2)}{2e}
(\mu_1|x|^{2^*}+\mu_2|y|^{2^*})
+\frac{2(N-2)}{N}|x|^{p}
|y|^{q} +\frac{|\Omega|}{e}(\mu_1+\mu_2)\right),
\end{equation}
the definition of $S(N)$ here is given in \eqref{eq:Sobolev_constant}.
Then any minimizing sequence associated to $c_\alpha$ is relatively compact in
$\mathcal{B}_\alpha$. In particular, $c_\alpha$ is achieved.
\end{theorem}

As we noticed, in the Sobolev-critical case, our results are also new even
for the single logarithmic Schr\"{o}dinge equation. In particular,
choosing $\rho_1=\rho_2=\rho$, $\mu_1=\mu_2=\mu$ and $u_1=u_2$, we have the following direct
consequence.
\begin{theorem}[Br\'{e}zis-Nirenberg problem: $p+q=2^*$]\label{critical}
	Let $\Omega\subset \R^N$  be a smooth bounded domain. If $\mu>0$ and
$$0<\rho\le\left(\frac{R(\Omega,N,2^*)}{\left(\frac{(N-2)}{2e}S(N)
+\frac{|\Omega|}{e}\right)\mu}\right)^{\frac{N-2}{2} },$$
then the problem
$$\left\{\aligned\ &\ -\Delta u + \omega u=\mu u \log u^2+|u|^{2^*-2}u,\\
\ &\ \int_\Omega u^2\,dx=\rho, \ \  u\in H^1_0(\Omega),
\endaligned\right.$$
admits a positive solution $u$, which is a local minimizer of the associated energy.
Moreover, the corresponding set of local ground states is orbitally stable.
\end{theorem}

Based on the previous theorem, we introduce the following set of admissible masses
\begin{equation}\label{eq:defA}
A:=\left\{(\rho_1,\rho_2)\in(0,\infty)^2:
\aligned\ &\  c_\alpha<\hat{c}_\alpha\ \text{for some}\ \alpha\ge\lambda_1(\Omega),\\
 \
&\ \text{with}\ \alpha \ \text{satisfying \eqref{eq:compact_intro} if }\ \
p+q=2^*\endaligned \right\}\cup\left\{(0,0)\right\}.
\end{equation}
		
Notice that, as a matter of fact, $A$ depends on $\Omega$, $N$, $p,\ q$,
$\mu_1$, $\mu_2$. Moreover, if
$(\rho_1,\rho_2)\in A$, then we can choose the local minimizer
$(u_1,u_2)\in\mathcal{M}$ to be a
positive
solution of \eqref{eq:1} for some $(\omega_1,\omega_2)\in
\mathbb R^2$.

\vskip0.13in

\begin{theorem}[Existence]\label{prop:supercritical}
	Let $\Omega\subset \R^N$  be a smooth bounded domain and $2\leq p+q\le 2^*$. If $A$ is defined as in \eqref{eq:defA},
then $A$ is star-shaped with respect to
$(0,0)$. Moreover, there exists a positive constant $R=R(\Omega,N,p+q)$ such that if
$\rho_1,\rho_2>0$ are such that
\begin{equation}\label{eq:assnice}
	\left(\frac{(N-2)}{2e}S(N)\max\{\mu_1,\mu_2\}+\frac{|\Omega|}{e}(\mu_1+\mu_2)\right)(\rho_1+\rho_2)^{a-1}\le R(\Omega,N,p+q),
\end{equation}
then $(\rho_1,\rho_2)\in A$. Here $a$ and $r$ are defined as in \eqref{eq:newexponent}
and $R$ is explicit, see \eqref{eq:def_R}.
\end{theorem}
\vskip0.13in	

\begin{theorem}[Stability]\label{thm:stab}
	Let $\Omega\subset \R^N$  be a smooth bounded domain, $2\leq p+q\le 2^*$ and
$(\rho_1,\rho_2)\in A$. Let
$\bar{\alpha}\ge\lambda_1(\Omega)$ be such that
$$c_{\bar{\alpha}}<\hat{c} _{\bar{\alpha}},\ \
\text{and}\ \ \bar{\alpha}\ \text{satisfies}\  \eqref{eq:compact_intro}\ \ \text{if}\ \
p+q=2^*.$$
Then the set of local ground states
\begin{equation}\label{eqGa}
G_{\bar{\alpha}}:=\left\{(u_1,u_2)\in H^1_0(\Omega;\mathbb C^2):\
(|u_1|,|u_2|)\in\mathcal{B} _{\bar{\alpha}},\
\mathcal{E}(u_1,u_2)=c_{\bar{\alpha}}\right\},
\end{equation}
is orbitally stable.
\end{theorem}
\vskip0.13in
\begin{remark}
The definition of orbital stability is given in Section \ref{sec4}. See \cite{Nonlinearity} and its
references for more orbital stability results of solutions to nonlinear
Schr\"{o}dinger equations or systems. Actually, Theorem \ref{thm:stab} can be viewed as an extension
of the result in \cite{APDE,Nonlinearity,CVPDE} for the Schr\"{o}dinger system to the logarithmic case.
\end{remark}
\vskip0.13in	
\begin{theorem}[Non-existence]\label{thm:5}
Let $\Omega\subset\R^N$ be a smooth bounded domain. If $\Omega$ is star shapped with respect to 
$0\in\R^N$ and $p+q<2$, then system \eqref{eq:1} has no solutions, except the trivial one.
\end{theorem}
\vskip0.13in	
\begin{remark}
Comparing the results of \cite{deng} and \cite{mma}, we provide a unified approach to study the
existence of positive solutions to the Br\'{e}zis-Nirenberg problem with logarithmic perturbation. Different from \cite{deng} and \cite{mma}, we not only prove the
existence of the normalized ground state
solutions, but also consider its orbital stability of the normalized solutions.
Furthermore, our results are also new even for the single logarithmic
 Schr\"{o}dinger
equation. Therefore, our results improve some previous work.
It is worth mentioning that we applied a different approach 
from the previous ones to study the Br\'{e}zis-Nirenberg problem. 
Inspired by \cite{Nonlinearity}, we do not rely here on the scaling
transformation on the manifold and the mountain path theorem.
\end{remark}

Now, we consider the behavior of solutions for subcritical problem on $\Omega=\R^N(N\geq3)$:
\begin{equation}\label{eq:106}
\left\{\aligned\ &\ -\Delta u_1+\omega_1u_1=\mu_1 u_1\log u_1^2+\frac{2p}{p+q}|u_2|^{q}|u_1|^{p-2}u_1,\\
\ &\ -\Delta u_2+\omega_2u_2=\mu_2 u_2\log u_2^2+\frac{2q}{p+q}|u_1|^{p}|u_2|^{q-2}u_2,\\
\ &\ (u_1,u_2)\in\mathcal{H} (\R^N),\ \ \omega_i\in\R,\ \mu_i>0,\ \ i=1,2.\endaligned\right.
\end{equation}
Moreover, by studying the behavior of system \eqref{eq:1} when $p+q$ tends to $2^*$, we
give a way to prove the existence of normalized solutions to the
Br\'{e}zis-Nirenberg problem on $\R^N(N\geq4)$.

\begin{theorem}[Behavior of solutions]\label{thm:6}
Assume $N\geq 3,\ p,\ q>1,\ 2<p+q<2^*$. System \eqref{eq:106} has a ground state solution
$(u_1,u_2)\in\mathcal{H}(\R^N)$ with
$$\mathcal{I}_{p+q}(u_1,u_2)=c_{p+q}:=\inf_{\mathcal{N}_{p+q}}\mathcal{I}_{p+q}(u_1,u_2),$$
where
$$ \mathcal{N}_{p+q}:=\{ (u_1,u_2)\in \mathcal{H}(\R^N)\backslash \{(0,0)\},\ \
\mathcal{J}_{p+q}(u_1,u_2)=0\},$$
$$\aligned\mathcal{I}_{p+q}(u_1,u_2):&=
\frac{1}{2}\int_{\R^N}(|\nabla u_1|^2+|\nabla u_2|^2+\omega_1 u_1^2+\omega _2 u_2^2-\mu_1 u_1^2(\log u_1^2-1)-\mu_2 u_2^2(\log u_2^2-1))\\
&\ \ \ \ -\frac{2}{p+q}\int_{\R^N}|u_1|^{p}|u_2|^{q},
\endaligned$$
and
$$\aligned\mathcal{J}_{p+q}(u_1,u_2):&=\int_{\R^N}\left(|\nabla u_1|^2+|\nabla u_2|^2
+\omega_1 u_1^2+\omega_2 u_2^2
-\mu_1 u_1^2\log u_1^2-2|u_1|^{p}|u_2|^{q}-\mu_2 u_2^2\log u_2^2\right).
\endaligned$$
For $N\geq 4$, let $p+q$ approaches to $2^*$, system
\begin{equation}\label{eq:1071}
\left\{\aligned\ &\ -\Delta u_1+\omega_1u_1
=\mu_1 u_1\log u_1^2+\frac{2p}{2^*}|u_2|^{q}|u_1|^{p-2}u_1,\\
\ &\ -\Delta u_2+\omega_2u_2
=\mu_2 u_2\log u_2^2+\frac{2q}{2^*}|u_1|^{p}|u_2|^{q-2}u_2,\\
\ &\ (u_1,u_2)\in  \mathcal{H} (\R^N),\ \ \omega_i\in \R,\ \mu_i>0,\ \ i=1,2,\\
\ &\ p,\ q>1,\ \ p+q=2^*,\endaligned\right.
\end{equation}
has a ground state solution $(u_1,u_2)\in\mathcal{H}(\R^N)$ with
$$\mathcal{I}_{2^*}(u_1,u_2)=c_{2^*}=\inf_{\mathcal{N}_{2^*}}\mathcal{I}_{2^*}(u_1,u_2).$$
\end{theorem}

\begin{remark}
When we prove Theorem 
\ref{thm:6}, we only use 
the Nehari constraint set. Note that it is not a manifold, because 
the existence of the
 logarithmic term causes the functional $\mathcal{I}$ losses $C^1$ 
 smooth on $\mathcal{H}(\mathbb R^N)$, but we can first prove 
 that the minimum value of $\mathcal{I}_{p+q}|_{\mathcal{N}_{p+q}}$ is reachable,
and then obtain that the minimum of 
$\mathcal{I}_{p+q}|_{\mathcal{N}_{p+q}}$ is a 
solution by the method in \cite[Lemma 2.5]{lww}. On the other hand, 
the above method is not available when working in 
bounded domains. Moreover, the gain of compactness is lost
when we face the Sobolev critical case.
\end{remark}

The paper is structured as follows. In the next (i.e., Section \ref{sec2}) we make some preliminary
remarks and notations which will be used in the text. Section \ref{sec3}
is devoted to the existence results, i.e., to the proof of Theorems
 \ref{prop:compact_intro}-\ref{prop:supercritical}. The stability results,
namely the proof of Theorem \ref{thm:stab} is proved in Section \ref{sec4}.
Theorem \ref{thm:5} is proved in section \ref{sec5}. Finally, we investigate the behavior of
solutions
for problem \eqref{eq:1} when $p+q$ tends to $2^*$ and complete the proof of Theorem \ref{thm:6} in
section \ref{sec:segregation6}.

\section{Notations and Preliminary results}\label{sec2}

Throughout the paper, we do not distinguish between bounded domain $\Omega$ and $\Omega=\R^N$, unless otherwise specified.
We denote $\lambda_1(\Omega)$ as the first eigenvalue of the Dirichlet Laplacian in $\Omega$, and by
$\varphi_1$ the corresponding first eigenfunction, which we assume
normalized in $L^2(\Omega)$ and positive in $\Omega$. $|\Omega|$ denotes the measure of $\Omega$. For
convenience, $C$ and $C_i (i = 1, 2,\ldots)$ denote (possibly different) positive
constants. $B_R(x)$ is a ball with radius $R$ centered on $x$. $\int_\Omega g(z)$ denotes the integral
$\int_\Omega g(z)dz$. The $ \rightarrow$ and $\rightharpoonup$ denote strong
convergence and weak convergence, respectively.
We use the following $L^q(\Omega)$ ($1\le q<\infty$), $H^1_0(\Omega)$  and $H^1(\R^N)$-norms:
$$|u|_{L^q(\Omega)}^q:=\int_\Omega|u|^q, \ \  \|u\|_{H^1_0(\Omega)}^2:
=\int_\Omega|\nabla u|^2,\ \ \|u\|_{H^1(\R^N)}^2:=\int_{\R^N}(|\nabla u|^2+|u|^2).$$
Where there is no risk of confusion, we will denote
$|\cdot|_{L^q(\Omega)}$ simply by $|\cdot|_q$.

We shall make frequent use of the following lemma in which we summarize some
properties of the function $s^2\log s^2$.

\begin{lemma}\label{lemma:1}
The function $s^2\log s^2$ satisfies the following properties:

$(i)$ $|s^2\log s^2|\leq\frac{1}{e},\  0\leq s\leq1;$

$(ii)$ $0<s^2\log s^2\leq\frac{(N-2)}{2e}s^{2^*},\ s>1$.
\end{lemma}
\begin{proof}
(i) Define a function in $[0,+\infty)$ by $f(s)=(s^2\log s^2)^2$, by direct computation, we obtain
$$f^\prime(s)=2s^2(2s\log s^2+2s)\log s^2=4s^3(2\log s+1)\log s^2.$$
Let $f^\prime(s)=0$, we get $s=0$ or $s=e^{-\frac 12}$.
It is easy to see that $f>0$ for $0<s<e^{-\frac 12}$ and $f<0$ for $e^{-\frac 12}<s<1$. Hence $f(s)<f(e^{-\frac 12})=e^{-2}$,
that is, for $0\leq s\leq 1$, we obtain $|s^2\log s^2|\leq\frac{1}{e}$.

(ii) Since
$$\lim_{s\to\infty}\frac{s^{2^*}}{s^2\log s^2}=+\infty,$$
there exists a positive constant $\theta (N)$ such that $s^2\log s^2\leq \theta (N)s^{2^*}$ for $s>1.$
Let $\alpha=2^*-2$, we define $g(s):=\frac{\log s^2}{s^\alpha}$, by complete, we obtain $g(s)\leq\frac{(N-2)}{2e}$.
\end{proof}	

Next, we recall that a lemma by Br\'{e}zis and Lieb \cite{BL83}: for $1\leq q<\infty$,  if
$\{g_n\}_n\subset L^q(\Omega)$ is a sequence bounded in $L^q(\Omega)$, such that $g_n\to g$ almost
everywhere, then
\begin{equation}\label{eq:BL}
|g_n|_q^q=|g|_q^q+|g_n-g|_q^q+o(1)\ \ \text{as}\ \ n\to\infty.
\end{equation}

In the following, we give the Br\'{e}zis-Lieb type lemma for $u^2\log u^2$ 
(see \cite[Lemma 2.3]{ar} or \cite[Lemma 3.1]{mma22}) and the coupling
term $|u_1|^p|u_2|^q$ (see \cite[Lemma 3.2]{chenzhang}) respectively.

\begin{lemma}\label{le332}
Let $u_{n}\rightharpoonup u$ in $H_0^1(\Omega)$. If $u_{n}\rightarrow u$ a.e in $\Omega$,
$\{u_n^{2}\log u_n^{2}\}$ is a bounded sequence in $L^1(\Omega)$. Then,
$u^{2}\log u^{2}\in L^1(\Omega)$ and
$$\int_{\Omega}u_{n}^{2}\log u_{n}^{2}=\int_{\Omega}u^{2}\log u^{2}
+\int_{\Omega}|u_{n}-u|^{2}\log|u_{n}-u|^{2}+o(1)\ \ \text{as}\ \ n\to\infty. $$
\end{lemma}

\begin{proof}
For the convenience of the reader, we give the details of the proof which
was sketched in \cite[Lemma 3.1]{mma22}.
We define
$$A(s):=\left\{\aligned\ &\ -s^2\log s^2, &\text{if}\ 0\leq s\leq e^{-3},\\
\ &\ 3s^2+4e^{-3}s-e^{-6},&\text{if}\ s\geq e^{-3},\endaligned\right.$$
and $B(s):=s^2\log s^2+A(s).$
Clearly, $A$, $B$ are positive, convex, increasing functions on $(0,+\infty)$ with
$A(0)=0$, $B(0)=0$ and there exists $C_\xi>0$ such that
\begin{equation}\label{eq3.2}
A(2s)\leq 4A(s),\ \ |B(s)|\leq C_\xi|s|^\xi\ \ \text{for any}\ \ \xi\in(2,2^*).
\end{equation}
Since $\{u_n\}$ is bounded in $H_0^1(\Omega)$, we obtain that
$\{B(|u_n|)\}$ is bounded in $L^1(\Omega)$. Therefore, $\{A(|u_n|)\}$ is also
bounded in $L^1(\Omega)$. It follows from Fatou lemma that
$$\int_{\Omega}A(|u|)\leq\liminf_{n\to\infty}\int_{\Omega}A(|u_n|),\ \
\int_{\Omega}B(|u|)\leq\liminf_{n\to\infty}\int_{\Omega}B(|u_n|).$$
Thus,
$$\int_{\Omega}u^2\log u^2=\int_{\Omega}B(|u|)-\int_{\Omega}A(|u|)<+\infty,$$
which implies $u^2\log u^2\in L^1(\Omega)$.
Since $s^2\log s^2=B(s)-A(s)$, we need only apply the Br\'{e}zis-Lieb lemma
to the functions $A$ and $B$.
Obviously, by using \eqref{eq3.2}, according to that $\{A(|u_n|)\}$, $\{B(|u_n|)\}$ are
bounded in $L^1(\Omega)$, we obtain that $A$ and $ B$ satisfy \cite[ Theorem 2 and
example (b)]{BL83}.  Therefore,
$$\lim_{n\to\infty}\int_{\Omega}A(|u_n|)-A(|u_n-u|)-A(|u|)=0,$$
and
$$\lim_{n\to\infty}\int_{\Omega}B(|u_n|)-B(|u_n-u|)-B(|u|)=0.$$
Then, Lemma \ref{le332} follows.
\end{proof}

\begin{lemma}\label{32}
Let  $u_{1,n}\rightharpoonup u_1, u_{2,n}\rightharpoonup u_2$
in $H_0^{1}(\Omega)$,
$u_{1,n}\rightarrow u_1, u_{2,n}\rightarrow u_2$ a.e in $\Omega$. Then
$$\lim_{n\rightarrow\infty}\int_{\Omega}|u_{1,n}|^{p}|u_{2,n}|^{q}
-\int_{\Omega}|u_1|^{p}|u_2|^{q}
=\lim_{n\rightarrow\infty}\int_{\Omega}|u_{1,n}-u_1|^{p}|u_{2,n}-u_2|^{q}.$$
\end{lemma}
\begin{proof}
The proof can be proved by the same argument as that of \cite[Lemma 3.2]{chenzhang}, so we omit the details.
\end{proof}
\begin{lemma}\label{le3321}
Let $u_{n}\rightharpoonup u$ be a bounded sequencein $H_0^1(\Omega)$ such that $u_{n}\rightarrow u$ a.e
in $\Omega$ as $n\to\infty$, then
\begin{equation}\label{eq:mixed_term_conv1}
\lim_{n\to\infty}\int_{\Omega}u_{n}^{2}\log u_{n}^{2}=\int_{\Omega}u^{2}\log u^{2}.
\end{equation}
\end{lemma}

\begin{proof}
Under the conditions, by using Lemma \ref{lemma:1}, there exists some $C > 0$ such that
$$ \left|\int_{\Omega}u_{n}^{2}\log u_{n}^{2}\right|\leq C,\ \ \ \
\left|\int_{\Omega}u^{2}\log u^{2}\right|\leq C.$$
We claim that there exist $C>0$ and $\delta>0$ such that
$|s^2\log s^2|\leq Cs^{2-\delta}+Cs^{2+\delta},\ s\in(0,+\infty)$. Indeed, here we can prove it by
$$\lim_{s\to0}\frac{s^2\log s^2}{s^{2-\delta}}=0\ \ \text{and}\ \ 
\lim_{s\to\infty}\frac{s^2\log s^2}{s^{2+\delta}}=0.$$ According to
that Lemma \ref{le332} and the embedding of $H_0^1(\Omega)\hookrightarrow L^{p+q}(\Omega)$
is compact, we get
$$ \left|\int_{\Omega}|u_n-u|^2\log |u_n-u|^2\right|\leq
C\int_{\Omega}|u_n-u|^{2-\delta}+C\int_{\Omega}|u_n-u|^{2+\delta}\to0 \ \ \text{as}\ \ n\to\infty.$$
Therefore, $$\lim_{n\to\infty}\int_{\Omega}u_{n}^{2}\log u_{n}^{2}=\int_{\Omega}u^{2}\log u^{2}.$$
\end{proof}
 	
\section{Existence of normalized solutions}\label{sec3}
					
Assume from now on that $p$ and $q$ satisfy $2\leq p+q\leq 2^*$.
In this section, we do not distinguish between Sobolev-critical and Sobolev-subcritical
cases, unless otherwise specified.

\begin{lemma}\label{lemma:geometry_mp}
$\mathcal{E}$ restricted to $\mathcal{M}$ is not coercive for every
$\rho_1,\ \rho_2>0$.
\end{lemma}
	
\begin{proof}
Using Lemma \ref{lemma:1}, H\"older inequality and Gagliardo-Nirenberg inequality \eqref{1.12} we have
\begin{equation}\label{eq:importantestimateG}
\aligned\ &\ \int_\Omega\left(\mu_1 u_1^2(\log u_1^2-1)
+\frac{4}{p+q}|u_1|^{p}|u_2|^{q}+\mu_2 u_2^2(\log u_2^2-1)\right)\\
\ =&\ \mu_1\left(\int_{u_1^2\geq1}u_1^2\log u_1^2
+\int_{u_1^2\leq1}u_1^2\log u_1^2\right)+\mu_2 \left(\int_{u_2^2\geq1} u_2^2\log u_2^2
+\int_{u_2^2\leq1}u_2^2\log u_2^2\right)\\
\ &\ + \ \frac{4}{p+q}\int_\Omega|u_1|^{p}|u_2|^{q}\\
\ \le&\ \mu_1\left(\int_{u_1^2\geq1}u_1^2\log u_1^2
+\int_{u_1^2\leq1}|u_1^2\log u_1^2|\right)
+ \mu_2\left(\int_{u_2^2\geq1}u_2^2\log u_2^2+\int_{u_2^2\leq1}|u_2^2\log u_2^2|\right)\\
\  &\  +\frac{4}{p+q}\int_\Omega|u_1|^{p}|u_2|^{q}\\
\ \le&\  \int_\Omega(\mu_1|u_1|^{2^*}+\mu_2|u_2|^{2^*})
+\frac{|\Omega|}{e}(\mu_1+\mu_2)+\frac{4}{p+q}|u_1|_{p+q}^{p}|u_2|_{p+q}^{q}\\
\ <&\ \frac{(N-2)}{2e}S(N)\left(\mu_1|\nabla u_1|_2^{2^*}
+\mu_2|\nabla u_2|_2^{2^*}\right)+\frac{4}{p+q}C(N,p+q)
\rho_1^{\frac{2rp}{p+q}}\rho_2^{\frac{2rq}{p+q}}|\nabla u_1|_2^{\frac{2ap}{p+q}}
|\nabla u_2|_2^{\frac{2aq}{p+q}}\\
\ &\ + \  \frac{|\Omega|}{e}(\mu_1+\mu_2).\endaligned
\end{equation}
Since when $p+q=2^*$, we can compute $a=\frac{N(p+q-2)}{4}=2^*,\ r=\frac{ p+q }{4}-\frac{N(p+q-2)}{8}=0.$
	As a consequence we have, for $(u_1,u_2)\in\mathcal{M}$,
	\begin{equation}\label{eq:importantestimate}
	\aligned\mathcal{E}(u_1,u_2)
	&>\frac{1}{2}(|\nabla u_1|_2^2+|\nabla u_2|_2^2)-\frac{1}{2}\bigg(\frac{(N-2)}{2e}S(N)
	\left(\mu_1|\nabla u_1|_2^{2^*}
	+\mu_2|\nabla u_2|_2^{2^*}\right)\\
	&\ \ \  \ +\frac{4}{p+q}C(N,p+q)
	\rho_1^{\frac{2rp}{p+q}}\rho_2^{\frac{2rq}{p+q}}
	|\nabla u_1|_2^{\frac{2ap}{p+q}}|\nabla u_2|_2^{\frac{2aq}{p+q}}+\frac{|\Omega|}{e}(\mu_1+\mu_2)\bigg)\\
	&\to-\infty\endaligned
	\end{equation}
	 as $\|(u_1,u_2)\|_{H_0^1(\Omega)}\to+\infty.$ This concludes the proof.
	\end{proof}

Recall that, for $\alpha\geq\lambda_1(\Omega)$, $\mathcal{B}_\alpha$ and
$\mathcal{U}_\alpha$ are defined in \eqref{eq:Balpha}, while $c_\alpha$ and $\hat c_\alpha$ are as in
\eqref{eq:calpha}. Notice that $\mathcal{B}_\alpha\neq \emptyset$, since it contains at least
$(\sqrt{\rho}_1\varphi_1,\sqrt{\rho}_2\varphi_2)$. Moreover
$$c_{\lambda_1(\Omega)}=\hat c_{\lambda_1(\Omega)}
=\mathcal{E}(\sqrt{\rho}_1\varphi_1,\sqrt{\rho}_2\varphi_2).$$
Recalling \eqref{eq:importantestimate} and the identification $x=|\nabla u_1|_2$ and $y=|\nabla u_2|_2$,
we are considering the function $\Phi:\R^2_+\to\mathbb R$ defined by
\begin{equation}\label{eq:vphi_def}
\Phi(x,y)=\frac{1}{2}\left(x^2+y^2-
\frac{(N-2)}{2e}S(N)(\mu_1 x^{2^*}+\mu_2y^{2^*})+\frac{4C(N,p+q)}{p+q}\rho_1^{\frac{2rp}{p+q}}
\rho_2^{\frac{2rq}{p+q}}
x^{\frac{2ap}{p+q}}y^{\frac{2aq}{p+q}}+\frac{|\Omega|}{e}(\mu_1+\mu_2)\right).
\end{equation}
Indeed, it follows from \eqref{eq:importantestimate} that
\begin{equation}\label{eq:Phi_ineq}
\mathcal{E}(u_1,u_2)\geq \Phi\left(|\nabla u_1|_2,|\nabla u_2|_2\right)\ \ \text{for any}\ \
(u_1,u_2)\in\mathcal{M}.
\end{equation}
In particular, this allows to estimate $\hat c_\alpha$ from below. To do that, we define the following
subsets of $\R^2$:
$$U_\alpha=\left\{(x,y)\in\mathbb R^2_+:\ x^2+y^2=(\rho_1+\rho_2)\alpha\right\}, \ \ V_\alpha=U_\alpha\cap
\left\{x\geq\sqrt{\rho_1\lambda_1(\Omega)},\ y\geq \sqrt{\rho_2\lambda_1(\Omega)}\right\}.$$
The set $U_\alpha$ is obtained from $\mathcal{U}_\alpha$ through the identification $x=|\nabla u_1|_2$
and $y=|\nabla u_2|_2$. The set $V_\alpha$ is motivated by the fact that, for $(u_1,u_2)\in\mathcal{M}$,
$|\nabla u_1|_2^2\geq\rho_1\lambda_1(\Omega)$ and $|\nabla u_2|_2^2\geq\rho_2\lambda_1(\Omega)$. From
\eqref{eq:Phi_ineq}, we deduce
\begin{equation}\label{eq:iniziostima}
\aligned \hat{c}_\alpha\ge\min_{(x,y)\in V_\alpha}\Phi(x,y)&=
\frac{1}{2}\bigg((\rho_1+\rho_2)\alpha- \max_{(x,y)\in V_\alpha}
\Bigl(\frac{(N-2)}{2e}S(N)(\mu_1 x^{2^*}+\mu_2y^{2^*})\\
&\ \ \ \ +\frac{4}{p+q}C(N,p+q)
\rho_1^{\frac{2rp}{p+q}}\rho_2^{\frac{2rq}{p+q}}x^{\frac{2ap}{p+q}}y
^{\frac{2aq}{p+q}}+\frac{|\Omega|}{e}(\mu_1+\mu_2)\Bigr)\bigg).\endaligned
\end{equation}
Now, due to the limitations in the definition of $V_\alpha$, the last maximum can not be explicitly
expressed by $\alpha$ (except for a few particular cases). Therefore, we prefer a more rough
estimate, where $V_\alpha$ is replaced by $U_\alpha$.

\begin{lemma}\label{lem:hatc_from_below}
For every $\alpha>\lambda_1(\Omega)$,
$$\hat c_\alpha>\frac12\left((\rho_1+\rho_2)\alpha-\Lambda(\rho_1,\rho_2)
(\rho_1+\rho_2)^{a}\alpha^{a}\right),$$
where
\begin{equation}\label{eq:defdd}
\aligned\Lambda=\Lambda(\rho_1,\rho_2):&=\max_{t\in[0,\frac{\pi}{2}]}
\bigg(\frac{(N-2)}{2e}S(N)((\rho_1+\rho_2)\alpha)^{\frac{2^*}{2}-a}(\mu_1\cos^{2^*}t+
\mu_2\sin^{2^*}t)\\
&\ \ \ \ +\frac{4}{p+q}C(N,p+q)\rho_1^{\frac{2rp}{p+q}}\rho_2^{\frac{2rq}{p+q}}
\cos^{\frac{2ap}{p+q}}t\sin^{\frac{2aq}{p+q}}t+\frac{|\Omega|}{e}(\mu_1+\mu_2)\bigg).
\endaligned
\end{equation}
\end{lemma}
		
\begin{proof}
Obviously, since
$$V_\alpha\subset U_\alpha=\left\{(\cos t,\sin t)\sqrt{(\rho_1+\rho_2) \alpha}:
t\in\left[0,\frac{\pi}{2}\right]\right\},$$
we conclude by \eqref{eq:iniziostima}.
\end{proof}
\begin{remark}
Observe that $\Lambda$ depends on $\mu_1,\mu_2$, and also on $p$ and
$N$ (via $a$ and $r$).
Furthermore, in case $p+q=2^*$, we know that $a=\frac{2^*}{2}$, $r=0$ and $\Lambda$ does not depend on
$\rho_1,\rho_2$, and actually its definition coincides with that given in
\eqref{eq:def_Lambda}. Therefore, according to the definition of $S_N$ given in
\eqref{eq:Sobolev_constant}, for every $(\tau_1,\tau_2)\in H^1_0(\Omega;\mathbb R^2)$, we obtain
\begin{equation}\label{eq:lambda_max_def}
\aligned\ &\ S(N)\left(\frac{(N-2)}{2e}
(\mu_1|\nabla\tau _1|_2^{2^*}+\mu_2|\nabla\tau_{2}|_2^{2^*})
+\frac{2(N-2)}{N}|\nabla\tau _1|_2^{p}
|\nabla\tau _2|_2^{q} +\frac{|\Omega|}{e}(\mu_1+\mu_2)\right)\\
\ \le&\ \Lambda\left(|\nabla\tau _1|_2^{2}
+|\nabla \tau _{2}|_2^2\right)^{\frac{2^*}{2}}.\endaligned
\end{equation}
To see this, for every $(\tau_1,\tau_2)$, one can find $t\in\left[0,\frac{\pi}{2}\right]$ such that
$$|\nabla\tau_1|_2=\left(|\nabla\tau_1|_2^{2}+|\nabla\tau_{2}|_2^2\right)^{\frac{1}{2}}
\cos t, \ \ |\nabla\tau_2|_2=\left(|\nabla\tau_1|_2^{2}+|\nabla\tau_{2}|_2^2\right)
^{\frac{1}{2}}\sin t,$$
and we substitute in \eqref{eq:defdd}.
\end{remark}
	
As mentioned earlier, we will look for local minimizers of $\mathcal{E} $ on
$\mathcal{B}_\alpha$, hence at level $c_\alpha$, for suitable values of $\alpha$. The
first necessary step is to prove that $c_\alpha$ is achieved. This is easily obtained,
for any $\alpha\ge\lambda_1(\Omega)$, in the Sobolev-subcritical case. Indeed,
in this case, $\mathcal{B}_\alpha$ is weakly compact and $\mathcal{E}$ weakly
lower semicontinuous, by using the method of the calculus of variations
(see \cite[Remark 2.3.3]{for beginer}), $c_\alpha$ is achieved by a couple $(u_1,u_2)$.

In the following, recall that $\Lambda$ given in
\eqref{eq:defdd} and that, in the Sobolev-critical case, it does not depend on $\rho_1,\rho_2$.
			
\begin{lemma}\label{prop:Sobcritconv}
Consider $\Omega\subset \R^N(N\geq 3)$ is a bounded smooth domain and $p+q=2^*$.  Let $\alpha>\lambda_1(\Omega)$, $\rho_1,\rho_2>0$ satisfy
$$(\rho_1+\rho_2)(\alpha-\lambda_1(\Omega))<\frac{1}{\Lambda^{\frac{N-2}{2}}},$$
and let $(u_{1,n},u_{2,n})_n$ be such that
\begin{equation}\label{eq:minimizing_seq}
\left\{\aligned&|u_{i,n}|_2^2=\rho_i+o(1), \ \ i=1,2, \\
&|\nabla u_{1,n}|_2^2+|\nabla u_{2,n}|_2^2\leq\alpha(\rho_1+\rho_2)+o(1), \ \ &\\
&c_\alpha\leq\mathcal{E}(u_{1,n},u_{2,n})\leq c_\alpha+o(1),\ \ &\endaligned\right.
\end{equation}
as $n\to\infty$. Then, up to subsequences,
$$(u_{1,n},u_{2,n})\to(\bar u_{1},\bar u_{2})\ \ \text{ in}\ \ \mathcal{H} (\Omega).$$
In particular, $c_\alpha$ is achieved.
\end{lemma}
		
\begin{proof}
According to the assumption, there exists $(\bar u_1,\bar u_2)\in H^1_0(\Omega;\mathbb R^2)$
such that, up to subsequences,
$$\left\{\aligned &|\bar u_i|_2^2=\rho_i,\ \ &i=1,2, \\
&u_{i,n}\rightharpoonup\bar u_i\ \  \text{in}\ \ H^1_0(\Omega),\ \ &i=1,2,\\
&|\nabla\bar u_i|_2^2\leq\liminf_{n\to\infty}|\nabla u_{i,n}|_2^2,\ \ &i=1,2.\endaligned\right.$$
Observe that $(\bar{u}_1,\bar{u}_2)$ is admissible for the minimization problem
$c_\alpha$, whence
\begin{equation}\label{eq:bar_u_i_admissible}
\mathcal{E}(\bar{u}_1,\bar{u}_2)\geq c_\alpha.
\end{equation}
Let $\tau_{i,n}=u_{i,n}-\bar u_i$ and notice that, for $i=1,2$,
\begin{equation}\label{eq:v_i_n}
\tau_{i,n}\rightharpoonup 0 \ \ \text{in}\ \ H^1_0(\Omega)\ \
\text{and}\ \ L^{2^*}(\Omega),
 \ \ \tau_{i,n}\to0 \ \ \text{in}\ \ L^2(\Omega).
\end{equation}
Note that the strong convergence of a subsequence of $(u_{1,n},u_{2,n})$ is
equivalent to the statement:
\begin{equation}\label{eq:BN_easy_case}
\text{ there exists a subsequence }\ \ (\tau_{1,n_k},\tau_{2,n_k})\ \ \text{such that}\ \
|\nabla \tau_{1,n_k}|_2^2+|\nabla \tau_{2,n_k}|_2^2\to 0.
\end{equation}
In such a case, by continuity of the Sobolev embeddings, we obtain
$c_\alpha=\mathcal{E}(\bar{u}_1,\bar{u}_2)$. Since a minimizing sequence for $c_\alpha$
exists, and it satisfies \eqref{eq:minimizing_seq}, we deduce that $c_\alpha$ is
achieved.
			
To conclude the proof, suppose by contradiction that \eqref{eq:BN_easy_case} does not
 hold, so that
\begin{equation}\label{eq:BN_difficult_case}
|\nabla \tau_{1,n}|_2^2+|\nabla \tau_{2,n}|_2^2\geq K>0.
\end{equation}
We can write
\begin{equation}\label{eq3.9}
\aligned\ &\ \mathcal{E}(u_{1,n},u_{2,n})\\
\ =&\ \frac{1}{2}(|\nabla(\bar{u}_1+\tau_{1,n})|_2^2
+|\nabla(\bar{u}_2+\tau _{2,n}|_2^2))+\frac{1}{2}\int_\Omega (\mu_1(\bar{u}_1
+\tau_{1,n})^2+\mu_2(\bar{u}_2+\tau_{2,n})^2)\\
&\ -\frac{1}{2}\int_\Omega\bigg(\mu_1(\bar{u}_1+\tau_{1,n})^2
\log(\bar{u}_1+\tau_{1,n})^2
+\frac{2(N-2) }{N}|(\bar{u}_1+\tau_{1,n})|^{p}
|(\bar{u}_2+\tau_{2,n})|^{q}\\
&\ +\mu_2(\bar{u}_2
+\tau_{2,n})^2\log(\bar{u}_2+\tau_{2,n})^2\bigg).\endaligned
\end{equation}
According to $u_{i,n}\rightharpoonup \bar u_i$ in $H^1_0(\Omega)$ and \eqref{eq:BL}, we have, for
 $i=1,2$,
\begin{equation}\label{eq:weak_conv}
|\nabla (\bar{u}_i+\tau_{i,n})|_2^2=
|\nabla \bar{u}_i|_2^2+|\nabla\tau _{i,n}|_2^2+o(1)\ \ \text{as}\ \ n\to\infty,
\end{equation}
and
\begin{equation}\label{eq:weak_conv2}
|\bar{u}_i+\tau_{i,n}|_2^2=|\bar{u}_i|_2^2+|\nabla\tau _{i,n}|_2^2+o(1)\ \ \text{as}\ \ n\to\infty.
\end{equation}
Using Lemmas \ref{le332} and \ref{32}, by replacing \eqref{eq:weak_conv}  and
 \eqref{eq:weak_conv2} into \eqref{eq3.9}, one has
$$\mathcal{E}(u_{1,n},u_{2,n})=\mathcal{E}(\bar u_1,\bar u_2)+\mathcal{E}
(\tau_{1,n},
\tau _{2,n})+o(1) \ \  \text{as}\ \ n\to\infty.$$
The last expression, together with \eqref{eq:minimizing_seq} and
\eqref{eq:bar_u_i_admissible}, implies
$$\mathcal{E}(\tau _{1,n},\tau _{2,n})\leq o(1) \ \  \text{as}\ \ n\to\infty,$$
whence, it follows from \eqref{eq:importantestimateG}, \eqref{eq:importantestimate}
(with $r=0$, $a=\frac{2^*}{2}$ and
 $C_{N,2^*}=S_N$) and \eqref{eq:lambda_max_def} that
$$\aligned &\ |\nabla\tau_{1,n}|_2^2+|\nabla\tau _{2,n}|_2^2\\
\leq&\
S(N)\left(
(\mu_1|\nabla\tau _1|_2^{2^*}+\mu_2|\nabla\tau_{2}|_2^{2^*})
+\frac{2(N-2)}{N}|\nabla\tau _1|_2^{p}
|\nabla\tau _2|_2^{q} +\frac{|\Omega|}{e}(\mu_1+\mu_2)\right)+o(1)\\
\leq&\ \Lambda\left(|\nabla\tau _{1,n}|_2^{2}+|\nabla\tau_{2,n}|_2^2\right)
^{\frac{N}{N-2}}+o(1).\endaligned$$
Next, we can use \eqref{eq:BN_difficult_case} to rewrite the last inequality as
$$\left(|\nabla \tau_{1,n}|_2^2+|\nabla \tau_{2,n}|_2^2\right)^{\frac{2}{N-2}}
\ge\frac{1}{\Lambda }+o(1).$$
Combining the previous inequality with \eqref{eq:weak_conv}, we obtain
$$\aligned\left(\frac{1}{\Lambda }+o(1)\right)^{\frac{N-2}{2} }
&\leq|\nabla \tau_{1,n}|_2^2
+|\nabla \tau _{2,n}|_2^2\\
&=|\nabla u_{1,n}|_2^2+|\nabla u_{2,n}|_2^2-(|\nabla\bar{u}_1|_2^2
+|\nabla \bar{u}_2|_2^2)+o(1)\\
&\le(\rho_1+\rho_2)\alpha-\lambda_1(\Omega)(\rho_1+\rho_2)+o(1),
\endaligned $$
as $n\to+\infty$, which contradicts the assumption.
\end{proof}	

\begin{proof}[Proof of Theorem \ref{prop:compact_intro}]
The theorem follows as a particular case of Lemma \ref{prop:Sobcritconv}, when
in \eqref{eq:minimizing_seq} both $|u_{i,n}|_2^2=\rho_i$, $i=1,2$, and
$|\nabla u_{1,n}|_2^2+|\nabla u_{2,n}|_2^2\leq\alpha(\rho_1+\rho_2)$.
\end{proof}

\begin{lemma}\label{lem:c<hat_c}
Assume that $\rho_1,\rho_2>0$ are such that, for some $\alpha_1,\alpha_2$,
$$\lambda_1(\Omega)\le\alpha_1<\alpha_2\ \ \text{and}\ \
\hat c_{\alpha_1}<\hat c_{\alpha_2};$$
furthermore, in the Sobolev-critical case $p+q=2^*$, we also
assume that
$$\alpha_2<\lambda_1(\Omega) +\frac{\Lambda ^{-\frac{N-2}{2}}}{\rho_1+\rho_2},$$
where $\Lambda$ is defined in \eqref{eq:def_Lambda}. Then $c_{\alpha_2}<\hat c_{\alpha_2}$, and $c_{\alpha_2}$ is achieved by a positive
solution of \eqref{eq:1}.
\end{lemma}

\begin{proof}
Firstly, $c_{\alpha_2}$ is achieved by some $(\bar {u}_1,\bar{u}_2)\in\mathcal{B}_{\alpha_2}$: as we already
observed, this is trivial in the Sobolev-subcritical case,
while in the critical one it follows by
Lemma \ref{prop:Sobcritconv}. Next we observe that
$$c_{\alpha_2} = \min\left\{\hat c_\alpha:\lambda_1(\Omega)\le \alpha \le
\alpha_2\right\}\le\hat c_{\alpha_1} < \hat c_{\alpha_2}.$$
Therefore, $(\bar{u}_1,\bar {u}_2)\in\mathcal{B} _{\alpha_2}\setminus
\mathcal{U}_{\alpha_2}$, and the lemma
follows.
\end{proof}

We denote by $\lambda_2(\Omega)$ the second eigenvalue of $-\Delta$ in
$H^1_0(\Omega)$, and by $\varphi_2$ a corresponding eigenfunction.

\begin{lemma}\label{lem:estimate_c}
We have
$$(\sqrt{\rho_1}\varphi_1,\sqrt{\rho_2} \varphi_1)\in \mathcal{U} _{\lambda_1(\Omega)},
\ \
\left(\sqrt{\rho_1}\frac{\varphi_2^+}{|\varphi_2^+|_2},\sqrt{\rho_2}
\frac{\varphi_2^-}{|\varphi_2^-|_2}\right)\in \mathcal{U} _{\lambda_2(\Omega)}.$$
In particular, if $\varphi_1^2\geq \frac{e}{\min\{\rho_1,\rho_2\}},$ or
$\varphi_1^2<\frac{e}{\min\{\rho_1,\rho_2\}}$ and
$$\mu_1\rho_1(\log (\sqrt{\rho_1}\varphi_1)^{2}-1)
+\frac{4}{p+q}\rho_1^{\frac p2}\rho_2^{\frac q2}|\varphi_1|^{p-2}|\varphi_1|^{q}
+\mu_2\rho_2(\log(\sqrt{\rho_2}\varphi_1)^{2}-1)\geq0,\ \  i=1,2,$$
then
$$\hat c_{\lambda_1(\Omega)} \leq \frac{\rho_1+\rho_2}{2}\lambda_1(\Omega).$$
Moreover, if $\varphi_2^2\geq \frac{e}{\min\{\rho_1,\rho_2\}},$ or
$\varphi_2^2<\frac{e}{\min\{\rho_1,\rho_2\}}$ and
$$\mu_1\rho_1(\log (\sqrt{\rho_1}\varphi_2)^{2}-1)
+\frac{4}{p+q}\rho_1^{\frac p2}\rho_2^{\frac q2}|\varphi_2|^{p-2}|\varphi_2|^{q}
+\mu_2\rho_2(\log(\sqrt{\rho_2}\varphi_2)^{2}-1)\geq0,\ \  i=1,2,$$
then
$$\hat c_{\lambda_2(\Omega)} \leq \frac{\rho_1+\rho_2}{2}\lambda_2(\Omega).$$
\end{lemma}

\begin{proof}
By direct computations, the first assertion is direct. We claim that
$$\aligned\hat c_{\lambda_1(\Omega)}& \leq\mathcal{E}(\sqrt{\rho_1}\varphi_1,
\sqrt{\rho_2}\varphi_1)\\
&=\frac{\rho_1+\rho_2}{2}\lambda_1(\Omega)-\frac{1}{2}\int_\Omega\bigg(\mu_1(\sqrt{\rho_1}\varphi_1)^{2}(\log (\sqrt{\rho_1}\varphi_1)^{2}-1)\\
&\ \ \ \
+\mu_2(\sqrt{\rho_2}\varphi_1)^{2}(\log (\sqrt{\rho_2}\varphi_1)^{2}-1)+\frac{4}{p+q}|\sqrt{\rho_1}\varphi_1|^{p}|\sqrt{\rho_2}\varphi_1|^{q}\bigg),\\
&\leq\frac{\rho_1+\rho_2}{2}\lambda_1(\Omega).
\endaligned$$
Indeed, by assumption and Lemma \ref{lemma:1}, we obtain
$$ \int_\Omega\bigg(\mu_1(\sqrt{\rho_1}\varphi_1)^{2}(\log (\sqrt{\rho_1}\varphi_1)^{2}-1)
+\frac{4}{p+q}|\sqrt{\rho_1}\varphi_1|^{p} |\sqrt{\rho_2}\varphi_1|^{q}
+\mu_2(\sqrt{\rho_2}\varphi_1)^{2}(\log (\sqrt{\rho_2}\varphi_1)^{2}-1)\bigg)\geq0.$$
for every $\rho_1,\ \rho_2>0$.

On the other hand, we also get
$$\aligned\hat c_{\lambda_2(\Omega)}&\leq \mathcal{E}
\left(\sqrt{\rho_1}\frac{\varphi_2^+}{|\varphi_2^+|_{2}},\sqrt{\rho_2}\frac{\varphi_2^-}{|\varphi_2^-|_{2}}\right)\\
&\leq\frac{\rho_1+\rho_2}{2}\lambda_2(\Omega)-\frac{1}{2}\int_\Omega
\bigg(\mu_1\left(\sqrt{\rho_1}\frac{\varphi_2^+}{|\varphi_2^+|_{2}}\right)^{2}
\log \left(\left(\sqrt{\rho_1}\frac{\varphi_2^+}{|\varphi_2^+|_{2}}\right)^{2}-1\right)\\
&\ \ \ \ +\mu_2\left(\sqrt{\rho_2}\frac{\varphi_2^-}{|\varphi_2^-|_{2}}\right)^{2}
\log\left( \left(\sqrt{\rho_2}\frac{\varphi_2^-}{|\varphi_2^-|_{2}}\right)^{2}-1 \right)
+\frac{4}{p+q}\left|\sqrt{\rho_1}\frac{\varphi_2^+}{|\varphi_2^+|_{2}}\right|^{p}
\left|\sqrt{\rho_2}\frac{\varphi_2^-}{|\varphi_2^-|_{2}}\right|^{q}\bigg),\\
&\leq\frac{\rho_1+\rho_2}{2}\lambda_2(\Omega).\endaligned$$
\end{proof}

\begin{lemma}\label{thm:existence_bad_cond}
Let $\rho_1,\rho_2>0$ be such that
\begin{equation}\label{eq:mainassL}
\Lambda(\rho_1,\rho_2)\cdot(\rho_1+\rho_2)^{a-1}
\le\frac{(a-1)^{a-1}}{a^{a}}\lambda_j(\Omega)^{-(a-1)},
\end{equation}
where $j=1,2$ be as above in Lemma \ref{lem:estimate_c}. Let
$\bar \alpha=\frac{a}{a-1}\lambda_j(\Omega).$
$(\bar u_1,\bar u_2)\in \mathcal{B}_{\bar \alpha}\setminus \mathcal{U}_{\bar\alpha}$
 such that
$\mathcal{E}(\bar{u}_1,\bar{u}_2)=c_{\bar\alpha}$, which means that
$(\bar{u}_1,\bar{u}_2)$ is a local minimum of
$\mathcal{E}|_{\mathcal{M} }$, corresponding to a positive solution of \eqref{eq:1} for some
$(\omega_1,\omega_2)\in\R^2$.  Equivalently,
$(\rho_1,\rho_2)\in A$ as defined in \eqref{eq:defA}.
\end{lemma}
\begin{proof}
According to Lemma \ref{lem:c<hat_c}, our aim is to find $\bar{\alpha}>\lambda_j(\Omega)$ such that
\begin{equation}\label{eq:c<hat_c_app}
\hat{c}_{\lambda_j(\Omega)} < \hat{c}_{\bar{\alpha}}.
\end{equation}
Now, we look for a sufficient condition implying \eqref{eq:c<hat_c_app}.
By using Lemmas \ref{lem:hatc_from_below} and \ref{lem:estimate_c}, it is sufficient to find
$\bar{\alpha}>\lambda_j(\Omega)$ such that
$$\frac{\rho_1+\rho_2}{2} \lambda_j(\Omega)
\le \frac12\left(\bar{\alpha}(\rho_1+\rho_2) -\Lambda (\rho_1,\rho_2)
(\rho_1+\rho_2)^{a}\bar{\alpha}^{a} \right)$$
(by Lemma \ref{lem:hatc_from_below}, the right hand side is strictly less than $\hat c_{\bar{\alpha}}$).
Recall that $a=\frac{N(p+q-2)}{4}$, if $p+q>2+\frac 4N$, that is, $a>1$ and
\begin{equation}\label{eq:passa}
\Lambda(\rho_1,\rho_2)(\rho_1+\rho_2)^{a-1}\le\frac{\bar{\alpha}-
\lambda_j(\Omega)}{\bar{\alpha}^{a}}.
\end{equation}
Furthermore, if $2<p+q\leq 2+\frac 4N$, i.e., $0<a\leq 1$.
It is still possible to get \eqref{eq:passa} established.
Therefore, by a direct computation,
the best possible choice for the right hand side is
$$\max_{\alpha\ge\lambda_j(\Omega)}\frac{\alpha-\lambda_j(\Omega)}{\alpha^{a}}
=\frac{(a-1)^{a-1}}{a^{a}}\lambda_j(\Omega)^{-(a-1)},
\ \ \text{achieved by}\ \ \bar{\alpha}=\frac{a}{a-1}\lambda_j(\Omega).$$
This choice of $\bar{\alpha}$ is possible, since it makes \eqref{eq:passa}
equivalent to \eqref{eq:mainassL}, the assumption
of the lemma. Furthermore, it is clear that $\bar{\alpha}>\lambda_j(\Omega)$. Therefore, in order to apply
Lemma \ref{lem:c<hat_c} and conclude the
proof, we only need to check that,  the additional assumption
\begin{equation}\label{eq:passa2}
\bar\alpha<\lambda_1(\Omega)+\frac{1}{\Lambda ^{\frac{N-2}{2}}(\rho_1+\rho_2)}
\end{equation}
holds true. This is straightforward since, being $a=\frac{2^*}{2}$, it follows from
\eqref{eq:passa} that
$$\Lambda(\rho_1+\rho_2)^{\frac{2}{N-2}}\le\frac{\bar{\alpha}-
\lambda_j(\Omega)}{\bar{\alpha}^{\frac{N}{N-2}}}
<\frac{\bar{\alpha}-\lambda_1(\Omega)}{(\bar{\alpha}-\lambda_1(\Omega))^{\frac{N}{N-2}}}
=\frac{1}{(\bar{\alpha}- \lambda_1(\Omega))^{\frac{2}{N-2}}},$$
which is equivalent to \eqref{eq:passa2}.
\end{proof}

\begin{remark}
When $\rho_1=\rho_2,\ \mu_1=\mu_2,\ p=q,\ \omega_1=\omega_2$,
if $(\bar u_1,\bar u_2)$ solves \eqref{eq:1},
 the solution $(\bar u_1,\bar u_2)$
does not coincide with $(\sqrt{\rho_1}\varphi_1,\sqrt{\rho_2}\varphi_1)$. Indeed, this last pair solves
\eqref{eq:1} $\Longleftrightarrow $
$$\left\{\aligned\ &\ (\lambda_1(\Omega)+\omega_1)\sqrt{\rho_1}\varphi_1
=\mu_1(\sqrt{\rho_1}\varphi_1)^2\log(\sqrt{\rho_1}\varphi_1)^2+\frac{2p}{p+q}|\sqrt{\rho_2}\varphi_1|^{q}|\sqrt{\rho_1}\varphi_1|^{p-2}\sqrt{\rho_1}\varphi_1,\\
\ &\ (\lambda_2(\Omega)+\omega_2)\sqrt{\rho_2}\varphi_1=\mu_2(\sqrt{\rho_2}\varphi_1)^2\log(\sqrt{\rho_2}\varphi_1)^2
+\frac{2q}{p+q}|\sqrt{\rho_1}\varphi_1|^{p}|\sqrt{\rho_2}\varphi_1|^{q-2}\sqrt{\rho_2}\varphi_1.\endaligned\right.$$
$$\Longleftrightarrow $$
$$\left\{\aligned\ &\ (\lambda_1(\Omega)+\omega_1)
=\mu_1\sqrt{\rho_1}\varphi_1(\log \rho_1+\log\varphi_1^2)
+\frac{2p}{p+q}|\sqrt{\rho_2}\varphi_1|^{q}|\sqrt{\rho_1}\varphi_1|^{p-2},\\
\ &\ (\lambda_2(\Omega)+\omega_2)=\mu_2\sqrt{\rho_2}\varphi_1
(\log \rho_2+\log\varphi_1^2)
+\frac{2q}{p+q}|\sqrt{\rho_1}\varphi_1|^{p}|\sqrt{\rho_2}\varphi_1|^{q-2}.\endaligned\right.$$
Therefore, $\rho_1=\rho_2=0,\ \lambda_i(\Omega)=-\omega_i,\ i=1,2$.
\end{remark}

Following the method of \cite{Nonlinearity,ntvDCDS}, we extend the strategy of studying Sobolev critical case of
nonlinear Schr\"{o}dinger system to Br\'{e}zis-Nirenberg problem with logarithmic perturbation. In this
section it is convenient to make explicit the dependence of some quantities with respect to
$\rho_1,\rho_2$: in view of this, we write $c_\alpha(\rho_1,\rho_2)$, $\hat c_\alpha(\rho_1,\rho_2)$,
$\mathcal{B} _\alpha(\rho_1,\rho_2)$, $\mathcal{U}_\alpha(\rho_1,\rho_2)$.  For shorter notation,
we define
$$G(u_1,u_2):=\int_\Omega\bigg(\mu_1u_1^{2}
(\log u_1^{2}-1)+\frac{4}{p+q}|u_1|^{p}|u_2|^{q}
+\mu_2 u_1^{2}(\log u_1^{2}-1)\bigg),$$
and introduce the optimization problem
$$M_\alpha(\rho_1,\rho_2):=\sup_{\mathcal{U}_\alpha(\rho_1,\rho_2)}G.$$
Observe that,
$$\hat c_\alpha(\rho_1,\rho_2)=\frac{1}{2}\left(\alpha(\rho_1+\rho_2)-\frac{1}{p+1}
M_\alpha(\rho_1,\rho_2)\right),$$
and that $\hat c_\alpha(\rho_1,\rho_2)$ is achieved at $(u_1,u_2)\in
\mathcal{U}_\alpha(\rho_1,\rho_2)\Longleftrightarrow M_\alpha(\rho_1,\rho_2)$ is
achieved at  $(u_1,u_2)$.

\vskip0.12in

Fix, if every, $(\rho_1,\rho_2)\in A\setminus\{(0,0)\}$. By
definition of $A$, there exist
$\alpha>\lambda_1(\Omega)$ and $(\bar{u}_1,\bar{u}_2)\in\mathcal{B}_\alpha$,
a solution of \eqref{eq:1}, such that
$\mathcal{E} (\bar u_1,\bar u_2)=c_{\alpha}< \hat c_{\alpha}$, and $\alpha$ satisfies
\eqref{eq:compact_intro} in case $p+q=2^*$.
Note that the assumption $c_{\alpha}<\hat c_{\alpha}$
implies
$$\int_\Omega(|\nabla \bar{u}_1|^2+|\nabla\bar{u}_2|)^2<(\rho_1+\rho_2)\alpha,\ \ \text{ so that}\ \  \bar{\alpha}:=\frac{1}{\rho_1+\rho_2}\int_\Omega(|\nabla
\bar{u}_1|^2+|\nabla \bar{u}_2|^2)<\alpha.$$
As a consequence, $(\bar u_1,\bar u_2)\in\mathcal{U}_{\bar{\alpha}}(\rho_1,\rho_2)$
achieves $\hat c_{\bar{\alpha}}=c_\alpha$.

\begin{lemma}\label{lem:starshap1}
If $s>0$ then $(s\bar u_1,s\bar u_2)\in\mathcal{U}_{\bar{\alpha}}
(s^2\rho_1,s^2\rho_2)$ achieves
$$\hat c_{\bar\alpha}(s^2 \rho_1,s^2\rho_2)=\frac{s^2}2\bar{\alpha}
(\rho_1+\rho_2)-\frac{1}{p+1}G(s\bar u_1,s\bar u_2).$$
\end{lemma}

\begin{proof}
By a direct computation, this follows by noticing that
$$(u_1,u_2)\in\mathcal{U}_{\bar\alpha}(\rho_1,\rho_2)\Leftrightarrow
(su_1,su_2)\in\mathcal{U}_{\bar\alpha}(s^2\rho_1,s^2\rho_2),$$
with
$$G(su_1,su_2)=\int_\Omega\bigg(\mu_1 s^2u_1^{2}
(\log s^2u_1^{2}-1)+\frac{4}{p+q}s^{p+q}|u_1|^{p}|u_2|^{q}
+\mu_2 s^2u_2^{2}(\log s^2u_2^{2}-1)\bigg).$$
Then  the lemma follows.
\end{proof}

\begin{lemma}\label{lem:curvetta}
Let $s\in(0,1)$, $|t|$ small,
$$(U_1(t),U_2(t)):=\bigg(s\sqrt{\rho_1}
\frac{\bar u_1+t\tau_1}{|\bar u_1+t\tau_1|_2} ,
s\sqrt{\rho_2}\frac{\bar u_2+t\tau _2}{|\bar u_2+t\tau_2|_2}\bigg),$$
where $(\tau_1,\tau_2)\in H^1_0(\Omega,\R^2)$ be such that
$$\int_\Omega \bar u_1\tau_1=\int_\Omega \bar u_2\tau_2=0,\ \
 \int_\Omega (\nabla \bar u_1 \nabla \tau_1+\nabla\bar u_2\nabla\tau_2)<0, $$
and
$$\aligned \ &\ \int_\Omega
\bigg(\frac{2p}{p+q}s^{p+q-1}|\bar u_1|^{p-1}\tau_1
+\frac{2q}{p+q}s^{p+q-1}|\bar u_2|^{q-1}\tau_2+\tau_1\log s^2\bar u_1^2+\tau_2\log s^2\bar u_2^2\bigg)>0.\endaligned$$
Then $(U_1(t),U_2(t))\in\mathcal{M} _{s^2\rho_1,s^2\rho_2}$
for every $t$ and
$$\frac{d}{dt}\|(U_1(t),U_2(t))\|^2_{H^1_0(\Omega)}\bigg|_{t=0}<0,\ \ \   \ \frac{d}{dt}
\mathcal{E} (U_1(t),U_2(t))\bigg|_{t=0}<0.$$
\end{lemma}

\begin{proof}
By a simple computation,
$$ \int_{\Omega}\left(s\sqrt{\rho_i}
\frac{ \bar u_i+t\tau_i }{| \bar u_i+t\tau_i |_2}\right)^2=s^2\rho_i,\ \ i=1,2,$$
we have that
$(U_1(t),U_2(t))\in\mathcal{M} _{s^2\rho_1,s^2\rho_2}$
for every $t$, and that
$$\frac{d}{dt}(U_1(t),U_2(t))\bigg|_{t=0}=(s\tau_1,s\tau_2).$$
Then, by assumption we get
$$\frac{d}{dt}\|(U_1(t),U_2(t))\|^2_{H^1_0(\Omega)}\bigg|_{t=0}
=2s^2\int_\Omega(\nabla \bar u_1\nabla \tau_1+\nabla\bar u_2\nabla\tau_2)<0.$$

Furthermore,
$$\aligned\ &\ \frac{d}{dt}G(U_1(t),U_2(t))\bigg|_{t=0}\\
\ =&\ \frac{d}{dt}\int_\Omega\bigg(\mu_1 U_1(t)^{2}
(\log U_1(t)^{2}-1)+\frac{4}{p+q}|U_1(t)|^{p}
|U_2(t)|^{q}+\mu_2 U_1(t)^{2}(\log U_1(t)^2-1)\bigg)\bigg|_{t=0}\\
\ =&\  2s\int_\Omega
\bigg(\frac{2p}{p+q}s^{p+q-1}|\bar u_1|^{p-1}\tau_1
+\frac{2q}{p+q}s^{p+q-1}|\bar u_2|^{q-1}\tau_2+\tau_1\log s^2\bar u_1^2+\tau_2\log s^2\bar u_2^2\bigg)>0.\endaligned$$
Therefore, for any $0<s<1$ and $p,\ q>1$, we have
$$\aligned\ &\ \frac{d}{dt}\mathcal{E}(U_1(t),U_2(t))\bigg|_{t=0}\\
\ =&\ s^2\int_\Omega(\nabla \bar u_1\nabla \tau_1+\nabla\bar u_2\nabla\tau_2)\\
\ &\ -s\int_\Omega\bigg(\frac{2p}{p+q}s^{p+q-1}|\bar u_1|^{p-1}\tau_1
+\frac{2q}{p+q}s^{p+q-1}|\bar u_2|^{q-1}\tau_2+\tau_1\log s^2\bar u_1^2+\tau_2\log s^2\bar u_2^2\bigg)\\
\ <&\ 0\endaligned$$
by assumption.
\end{proof}

\begin{lemma}\label{prop_starshaped}
The set $A$ is star-shaped with respect to $(0,0)$, where $A$ is defined in \eqref{eq:defA}.
\end{lemma}

\begin{proof}
According to the notation of Lemma \ref{lem:curvetta}, we see that
$$(U_1(0),U_2(0))=(s\bar u_1,s\bar u_2)\in\mathcal{U}_{\bar\alpha}
(s^2\rho_1,s^2\rho_2),$$
then there
exist positive and small constants $\varepsilon,\tau$ such that
$$(U_1(\tau ),U_2(\tau))\in\mathcal{U} _{\bar\alpha-\varepsilon}(s^2\rho_1,s^2\rho_2).$$
It follows from Lemmas \ref{lem:starshap1} and \ref{lem:curvetta} that
$$\hat{c}_{\bar{\alpha}-\varepsilon}(s^2\rho_1,s^2\rho_2) \le \mathcal{E}
(U_1(\tau),U_2(\tau))<\mathcal{E}(U_1(0),U_2(0))=\hat{c}_{\bar{\alpha}}
(s^2\rho_1,s^2\rho_2).$$
Then we can apply Lemma
\ref{lem:c<hat_c}, with $\alpha_1=\bar\alpha-\varepsilon$ and $\alpha_2=\bar\alpha$,
obtaining that $(s^2\rho_1,s^2\rho_2)\in A$. Since this holds true for any $s\in(0,1)$.
This completes the proof.
\end{proof}

Next, we give an explicit estimate  for $\Lambda $. At this point, the main
assumption in Lemma \ref{thm:existence_bad_cond} is written in terms of the
function $\Lambda(\rho_1,\rho_2)$ defined in \eqref{eq:defdd}.
For the reader's convenience, recall that
$$\aligned \Lambda (\rho_1,\rho_2)& =\max_{t\in[0,\frac{\pi}{2}]}
\bigg(\frac{(N-2)}{2e}S(N)((\rho_1+\rho_2)\alpha)^{\frac{2^*}{2}-a}(\mu_1\cos^{2^*}t+
\mu_2\sin^{2^*}t)\\
&\ \ \ \ +\frac{4}{p+q}C(N,p+q)\rho_1^{\frac{2rp}{p+q}}\rho_2^{\frac{2rq}{p+q}}
\cos^{\frac{2ap}{p+q}}t\sin^{\frac{2aq}{p+q}}t+\frac{|\Omega|}{e}(\mu_1+\mu_2)\bigg),\endaligned$$
where
$$a=\frac{N(p+q-2)}{4}\in \left(0,\frac{N}{N-2}\right],\ \  r=\frac{p+q}{4}-\frac a2\in\left[0,\frac 1N\right).$$
Next,we are going  to prove Theorem
\ref{prop:supercritical} by showing that
condition \eqref{eq:assnice} in Theorem \ref{prop:supercritical}, with
$R=R(\Omega,N,p+q)$ defined as
\begin{equation}\label{eq:def_R}
R(\Omega,N,p+q)=
\frac{(a-1)^{a-1}}
{a^{a}} \lambda_j(\Omega)^{-(a-1)},
\end{equation}
implies assumption \eqref{eq:mainassL} in Lemma \ref{thm:existence_bad_cond}.

\begin{proof}[End of the proof of Theorem \ref{prop:supercritical}]
	It follows from Lemma \ref{prop_starshaped} that $A$ is star-shaped with respect
to $(0,0)$. We estimate $\Lambda $ from above.  Note
that, as $0<a\leq\frac{2^*}{2}$, we have
$$\aligned\Lambda (\rho_1,\rho_2)\le\Lambda '(\rho_1,\rho_2):&=\max_{t\in[0,\frac{\pi}{2}]}
\bigg(\frac{(N-2)}{2e}S(N)(\mu_1\cos^{2^*}t+\mu_2\sin^{2^*}t)\\
&\ \ \ \ +\frac{4}{p+q}C(N,p+q)\rho_1^{\frac{2rp}{p+q}}\rho_2^{\frac{2rq}{p+q}}
\cos^{\frac{2p}{p+q}}t\sin^{\frac{2q}{p+q}}t+\frac{|\Omega|}{e}(\mu_1+\mu_2)\bigg)\\
&=\left(\frac{(N-2)}{2e}S(N)\max\{\mu_1,\mu_2\}+\frac{|\Omega|}{e}(\mu_1+\mu_2)\right).
\endaligned$$
Indeed, the maximum in the definition of $\Lambda$ is achieved
at either $t=0$ or $t=\frac\pi2$.
Therefore,  according to \eqref{eq:assnice} and \eqref{eq:def_R}, we obtain \eqref{eq:mainassL}, so that we can apply Lemma \ref{thm:existence_bad_cond} to conclude.
\end{proof}

\vskip0.3in
\section{Orbital stability of normalized solutions}\label{sec4}

In this section, we will prove Theorems
\ref{critical} and \ref{thm:stab} about the  stability of the normalized solutions.
Our aim is to prove the stability of the
sets  $G_{\bar\alpha}$ defined in  \eqref{eqGa}, respectively.

Inspired by \cite[Section 4]{Nonlinearity}, we claim that a set $\mathcal{G}\subset H^1_0(\Omega;\C^2)$ is orbitally
stable if for every $\varepsilon >0$ there exists $\delta>0$ such that, whenever
$(\psi_1,\psi_2)\in H^1_0(\Omega;\C^2)$ satisfies
$\text{dist}_{H^1_0}((\psi_1,\psi_2),\mathcal{G})<\delta$,  then the solution $(\Psi_1(t,\cdot),\Psi_2(t,\cdot))$ of
$$\left\{\aligned\ &\ \text{i}\partial_t\Psi_1+\Delta\Psi_1
+\Psi_1(\mu_1\log\Psi_1^2+|\Psi_2|^{q}|\Psi_1|^{p-2})=0,\\
\ &\ \text{i}\partial_t\Psi_2+\Delta\Psi_2+\Psi_2
(\mu_2\log\Psi_2^2+|\Psi_1|^{p}|\Psi_2|^{q-2})=0,\\
\ &\ \Psi_i(0,\cdot)=\psi_i(\cdot),\ \ \Psi_i(t,\cdot)\in H^1_0(\Omega;\C^2),
\endaligned\right.$$
must be continuous when  $ 0\leq t<+\infty$ and moreover,
\begin{equation}\label{eq:os2}
\sup_{t>0}\text{dist}_{H^1_0}((\Psi_1(t,\cdot),\Psi_2(t,\cdot)),\mathcal{G})<\varepsilon,
\end{equation}
where $\text{dist}_{H^1_0}$ denoting the
${H^1_0}$-distance.

\vskip0.3in

%Actually, note that we prove (conditional) orbital stability, where the condition is that the solution of
%system \eqref{eq:1}, with Cauchy datum $(\psi_{1},\psi_{2})\in H^1_0(\Omega;\C^2)$, exists locally in time for a
%time interval which is
%uniform in $\|(\psi_{1},\psi_{2})\|_{H^1_0}$, and that $\mathcal{Q} $ and $\mathcal{E} $ are preserved along the solutions

\begin{lemma}\label{lemma:c_c'}
For $(u_1,u_2)\in G_{\bar\alpha}$, there exist $\theta_1,\theta_2\in\R$
such that
$(u_1,u_2)=(e^{i\omega_1\theta_1}|u_1|,e^{i\omega_2\theta_2}|u_2|)$. In particular,
\[\inf \left\{\mathcal{E} (u_1,u_2): (u_1,u_2) \in H^1_0(\Omega;\C^2),
(|u_1|,|u_2|)
\in\mathcal{B} _{\bar\alpha}\right\}=c_{\bar\alpha},\]
while
\[\inf \left\{\mathcal{E} (u_1,u_2): (u_1,u_2) \in H^1_0(\Omega;\C^2), (|u_1|,|u_2|)
\in\mathcal{U} _{\bar\alpha}\right\} =:
\tilde c_{\bar \alpha}\le
\hat c_{\bar \alpha},\]
where $\mathcal{B} _{\bar\alpha}$ and $\mathcal{U} _{\bar\alpha}$ are defined in \eqref{eq:Balpha0} and \eqref{eq:Balpha}  respectively.
\end{lemma}

\begin{proof}
Given $(\tau_1,\tau_2)\in
\mathcal{U}_{\bar\alpha}$, according to the definition of $\mathcal{U}_{\bar\alpha}$ in \eqref{eq:Balpha}, it is easy to know that $(\tau_1,\tau_2)
\in H^1_0(\Omega;\C^2)$ and that $(|\tau_1|,|\tau_2|)
\in\mathcal{U}_{\bar\alpha}$, so that $\tilde c_{\bar\alpha}\le \hat c_{\bar\alpha}$. For
 $(u_1,u_2)\in G_{\bar\alpha}$, by using the diamagnetic inequality
 \cite[Theorem 7.21]{LiebLoss}, we obtain
$$\int_\Omega|\nabla|u_i||^2\le\int_\Omega|\nabla u_i|^2,$$
for $i=1,2$. Together with  $|u_i|^2\log|u_i|^2=u_i^2\log u_i^2,\ i=1,2$ and
$$c_{\bar\alpha}\le\mathcal{E}(|u_1|,|u_2|)\le\mathcal{E}(u_1,u_2)=c_{\bar\alpha}.$$
Therefore,
$$\int_\Omega|\nabla|u_i||^2=\int_\Omega|\nabla u_i|^2,\ \ i=1,2,$$
so that equality
holds in the
diamagnetic inequality, whence $u_i$ is a complex multiple of $|u_i|$,
that is to say $u_i=e^{i\omega_i\theta_i}|u_i|$ for some $\theta_i\in\R$, and the rest of
 the lemma follows.
\end{proof}

\begin{lemma}\label{lem:stability_compact}
Let  $\bar \alpha$ be defined  as above.  Assume that  $\{(\psi_{1,n},\psi_{2,n})\}\subset H^1_0(\Omega;\C^2)$ satisfies:
\begin{equation}\label{eq:stability_compact1}
\int_\Omega|\psi_{i,n}|^2\to\rho_i\ \ \ \  \text{for}\ \ i=1,2,\ \ \mathcal{E}(\psi_{1,n},\psi_{2,n})\to
c_{\bar\alpha}\ \ \text{as}\ \ n\to\infty,
\end{equation}
and, for $n$  large,

\begin{equation}\label{eq:stability_compact2}
\int_\Omega(|\nabla\psi_{1,n}|^2+|\nabla \psi_{2,n}|)^2\leq (\rho_1+\rho_2)\bar \alpha +\text{o}(1).
\end{equation}
Then there exists $(u_1,u_2)\in G_{\bar \alpha}$ such that, up to a subsequence,
$(\psi_{1,n},\psi_{2,n})\to (u_1,u_2)$ in $H^1_0(\Omega;\C^2)$.
\end{lemma}

\begin{proof}
Using \eqref{eq:stability_compact2} there exists $(\bar\psi_1,\bar\psi_2)\in H^1_0(\Omega;\C^2)$ such that, up
to a subsequence, $\psi_{i,n}\rightharpoonup u_i$ in $H^1_0(\Omega;\C)$ and $\psi_{i,n}\to v_i$ in
$L^2(\Omega;\C)$ for $i=1,2$, as $n\to+\infty$. Then \eqref{eq:stability_compact1},
\eqref{eq:stability_compact2} and Lemma \ref{lemma:c_c'} implies that
\[\int_\Omega|u_i|^2 =\rho_i,\ \
\int_\Omega(|\nabla u_1|^2+|\nabla u_2|^2)\leq(\rho_1+\rho_2)\bar\alpha,\ \
\mathcal{E}(u_{1},u_{2})\ge c_{\bar\alpha}, \quad i=1,2.\]
On the one hand, in case of $p+q<2^*$, the compact embedding  implies that $(\psi_{1,n},\psi_{2,n})\to(v_1,v_2)$ also in
$L^{p+q}(\Omega;\C^2)$. Therefore, it follows from Lemmas \ref{le332} and \ref{le3321} that
$$\int_\Omega|\psi_{i,n}|^{2}\log|\psi_{i,n}|^{2}=|u_{i}|^{2}\log|u_{i}|^{2},\ \
\int_\Omega|\psi_{1,n}|^{p}|\psi_{2,n}|^{q}=|u_{1}|^{p}|u_{2}|^{q},\ \ i=1,2.$$
According to the weak lower semicontinuity of the norm,
\[c_{\bar\alpha}\le\mathcal{E}(u_{1},u_{2})\le\liminf_{n\to+\infty}
\mathcal{E}(\psi_{1,n},\psi_{2,n})=c_{\bar\alpha},\]
together with the fact $(u_1,u_2)\in G_{\bar\alpha}$,we get the strong $H^1_0$-convergence.

\vskip0.1in
On the other hand, in case of $p+q=2^*$, the result follows  from Lemma \ref{prop:Sobcritconv}:
although that lemma is stated for real valued functions, combining with  Lemma \ref{lemma:c_c'},
its proof also holds for complex valued ones.
\end{proof}

\begin{lemma}\label{prop:stability}
Let $\bar \alpha$ be as above. If $c_{\bar \alpha}<\tilde c_{\bar \alpha}$, then $G_{\bar \alpha}$ is
(conditionally) orbitally stable.
\end{lemma}

\begin{proof}
Suppose by contradiction that $\{(\psi_{1,n},\psi_{2,n})\}\subset H^1_0
(\Omega;\C^2)$, $(u_{1,n},u_{2,n})\in G_{\bar \alpha}$ and $\bar \varepsilon >0$ are such that
\begin{equation}\label{eq:psi_to_u}
\lim_{n\to\infty} \|(\psi_{1,n},\psi_{2,n})-(u_{1,n},u_{2,n})\|_{H^1_0(\Omega;\C^2)}=0,
\end{equation}
and
\begin{equation}
\sup_{t>0}\text{dist}_{H^1_0}((\Psi_{1,n}(t,\cdot),\Psi_{2,n}(t,\cdot)),G_{\bar \alpha})
\geq 2\bar\varepsilon ,
\end{equation}
where $(\Psi_{1,n},\Psi_{2,n})$ is the solution of \eqref{eq:system_schro} with initial
condition $(\psi_{1,n},\psi_{2,n})$. Then there exists $\{t_n\}$ such that, letting
$\phi_{i,n}(x):=\Psi_{i,n}(t_n,x)$, $i=1,2$,
\begin{equation}\label{eq:stability_phi}
\text{dist}_{H^1_0}((\phi_{1,n},\phi_{2,n}),G_{\bar \alpha}) \geq\bar\varepsilon .
\end{equation}
Next, we prove that $\{(\phi_{1,n},\phi_{2,n})\}$ satisfies \eqref{eq:stability_compact1}
and \eqref{eq:stability_compact2}. Then Lemma \ref{lem:stability_compact} provides
a contradicton to \eqref{eq:stability_phi}. The proof of the lemma is now complete.

Lemma \ref{lem:stability_compact} implies that $G_{\bar \alpha}$ is compact. Therefore, by using
\eqref{eq:psi_to_u}, there exists $(u_1,u_2)\in
G_{\bar \alpha}$ such that, up to a subsequence,
\begin{equation}\label{eq:strongconvergence_stability}
 (\psi_{1,n},\psi_{2,n})\to (u_{1},u_{2}) \ \  \text{in}\ \  H^1_0(\Omega;\C^2).
\end{equation}
Combining  with the continuity of the  Sobolev embeddings,
it follows from Lemmas \ref{le332} and \ref{le3321} that
$$ \int_\Omega \psi_{1,n}^{2}\log \psi_{i,n}^2
=\int_{\Omega}u_i^2\log u_i^2,\ \
\int_{\Omega}|\psi_{1,n}|^p|\psi_{2,n}|^q=\int_{\Omega}|u_{1}|^{p}|u_{2}|^{q},\ \ i=1,2,$$
which means that
$(\psi_{1,n},\psi_{2,n})$ satisfies \eqref{eq:stability_compact1}. Then the
conservation of the mass and of the energy imply that
\[\int_\Omega|\phi_{i,n}|^2=\int_\Omega|\psi_{i,n}|^2\to\rho_i,\  \ i=1,2,
\ \ \text{and}\ \ \mathcal{E}(\phi_{1,n},\phi_{2,n})
=\mathcal{E}(\psi_{1,n},\psi_{2,n})\to c_{\bar\alpha},\]
as $n\to+\infty$,
so that $(\phi_{1,n},\phi_{2,n})$ also satisfies \eqref{eq:stability_compact1}. Next, we check that, at least for a subsequence, $(\phi_{1,n},
\phi_{2,n})$ satisfies \eqref{eq:stability_compact2} , i.e.,
\begin{equation}\label{eq:stability_contr2}
\int_\Omega (|\nabla \phi_{1,n}|^2 + |\nabla \phi_{2,n}|^2) \leq(\rho_1+\rho_2) \bar\alpha + \text{o}(1).
\end{equation}
Indeed, by contradiction, assume there exists $\bar n\in\N$ and $\bar\varepsilon>0$ such that
\[\int_\Omega(|\nabla\phi_{1,n}|^2+|\nabla\phi_{2,n}|^2) \geq(\rho_1+\rho_2)\bar\alpha+\bar\varepsilon.\]
Since
$$\aligned\int_\Omega(|\nabla\Psi_{1,n}(0,\cdot)|^2+|\nabla\Psi_{2,n}(0,\cdot)|^2)&=
\int_\Omega(|\nabla\psi_{1,n}|^2+|\nabla\psi_{2,n}|^2)\\
&\leq\int_\Omega(|\nabla u_{1,n}|^2+|\nabla u_{2,n}|^2)+\text{o}(1) \\
&\leq(\rho_1+\rho_2)\bar\alpha + \text{o}(1)\endaligned$$
for every $n$ sufficiently large, then there exists $\bar t_n\in (0,t_n)$ such that $(\Psi_{1,n}
(\bar t_n,\cdot), \Psi_{2,n}(\bar t_n,\cdot))$ satisfies \eqref{eq:stability_compact1}
and
\[\int_\Omega (|\nabla \Psi_{1,n}(\bar t_n,\cdot)|^2 + |\nabla \Psi_{2,n}
(\bar t_n,\cdot)|^2) =  (\rho_1+\rho_2)\bar \alpha + \text{o}(1),\]
and in particular \eqref{eq:stability_compact2}. According to Lemma
\ref{lem:stability_compact}
there exists $(\bar u_1,\bar u_2)\in G_{\bar \alpha}$ such that
\[\int_\Omega (|\nabla \bar u_{1}|^2 + |\nabla\bar u_2|^2)
=(\rho_1+\rho_2)\bar \alpha,\]
which contradicts the assumption $c_{\bar\alpha}<\tilde c_{\bar \alpha}$.
\end{proof}

We have proved Lemma \ref{prop:stability} by assuming that $c_{\bar \alpha}
<\tilde c_{\bar \alpha}$. Next  we check that this assumption is satisfied due to the fact  $c_{\bar \alpha}<\hat c_{\bar \alpha}$.

\begin{lemma}\label{lem:c_tilde}
Let $\bar \alpha$ be as above. Then $c_{\bar \alpha}<\tilde c_{\bar\alpha}$.
\end{lemma}

\begin{proof}
If by contradiction $\tilde c=c$, then there exists $\varepsilon _n\to 0$ and
$(\tau_{1,n},\tau_{2,n})\in H^1_0(\Omega;\C^2)$ such that
\[\|(\tau_{1,n},\tau_{2,n})\|^2_{H^1_0(\Omega;\C^2)}=\bar\alpha(\rho_1+\rho_2),\ \
\int_\Omega|\tau_{i,n}|^2=\rho_i, $$
and
$$c_{\bar \alpha}\leq
\mathcal{E}(\tau_{1,n},\tau_{2,n})\leq c_{\bar\alpha}+\varepsilon_n,\ \ i=1,2\]
for every $n$. Denoting $u_{i,n}:=|\tau _{i,n}|$, $i=1,2$, the diamagnetic inequality implies
\begin{equation}\label{eq:c_tilde1}
\|(u_{1,n},u_{2,n})\|^2_{H^1_0(\Omega;\R^2)} \leq
\|(\tau_{1,n},\tau_{2,n})\|^2_{H^1_0(\Omega;\C^2)}=(\rho_1+\rho_2)\bar\alpha,
\end{equation}
so that $(u_{1,n},u_{2,n})$ is an admissible couple for the minimization problem $c_{\bar\alpha}$
and then
\begin{equation}\label{eq:c_tilde2}
c_{\bar\alpha}\leq\mathcal{E}(u_{1,n},u_{2,n})\leq\mathcal{E}(\tau_{1,n},\tau _{2,n})
\leq c_{\bar\alpha}+\varepsilon_n.
\end{equation}
In particular,
\begin{equation}\label{eq:c_tilde3}
\frac{1}{2}\left(\|(\tau_{1,n},\tau _{2,n})\|^2_{H^1_0(\Omega;\C^2)}-
\|(u_{1,n},u_{2,n})\|^2_{H^1_0(\Omega;
\R^2)}\right)=\mathcal{E}(\tau_{1,n},\tau_{2,n})-\mathcal{E}
(u_{1,n},u_{2,n})\leq \varepsilon _n.
\end{equation}
Then Lemma \ref{lem:stability_compact} applies to both sequences, yielding both
$(\tau_{1,n},\tau_{2,n})\to(\tau_{1,\infty},\tau_{2,\infty})$ and
$(u_{1,n},u_{2,n})\to(u_{1,\infty},u_{2,\infty})$ in $H^1_0$. Passing to
the limit in \eqref{eq:c_tilde2} and \eqref{eq:c_tilde3}. This, combined with the
continuity of Sobolev embeddings and \eqref{le332} and \eqref{le3321}, we infer
$$\int_\Omega|u_{i,n}|^{2}\log|u_{i,n}|^{2}=|u_{1,\infty}|^{2}\log|u_{1,\infty}|^{2},\ \ i=1,2,$$
$$\int_\Omega|u_{1,n}|^{p}|u_{2,n}|^{q}=|u_{1,\infty}|^{p}|u_{2,\infty}|^{q},$$
and
\[\mathcal{E}(u_{1,\infty},u_{2,\infty})=c_{\bar\alpha}\ \ \text{and}\ \
\|(u_{1,\infty},u_{2,\infty})\|^2_{H^1_0(\Omega;\R^2)}=
\|(\tau_{1,\infty},\tau_{2,\infty})\|^2_{H^1_0(\Omega;\C^2)}=\bar\alpha(\rho_1+\rho_2).\]
Then $(u_{1,\infty},u_{2,\infty})\in\mathcal{U}_{\bar\alpha}$, contradicting the fact that $c_{\bar\alpha}<\hat c_{\bar \alpha}$.
\end{proof}

\begin{proof}[End of the proof of Theorems \ref{critical} and
\ref{thm:stab}] Recall that the first paragraph of this section, we have to
prove that the set $G_{\bar\alpha}$ is (conditionally) orbitally stable. This is a
direct consequence of Lemma \ref{prop:stability} together with Lemma
\ref{lem:c_tilde}.
\end{proof}

\section{Non-existence results}\label{sec5}

\begin{lemma}\label{le51}
If $u_1,u_2\in C^2(\Omega)\cap C^1(\overline{\Omega} )$ and $(u_1,u_2)$ is a solution of the
system \eqref{eq:1}, then $(u_1,u_2)$ satisfies $P(u_1,u_2)=0,$ where
\begin{equation}\label{0510}
\aligned P(u_1,u_2):&=\frac{N-2}{2} \int_{\Omega}(|\nabla u_1|^2+|\nabla u_2|^2)dx+\frac{N}{2}
\int_{\Omega}(\omega_1u_1^2+\omega_2 u_2^2)\,dx\\
 & =\frac{N}{2}\int_{\Omega}(\mu_1 u_1^2(\log u_1^2-1)+\mu_2 u_2^2(\log u_2^2-1))+
\frac{2N}{p+q}\int_{\Omega}|u_1|^p|u_2|^q\,dx\\
&\ \ \ \ +\int_{\partial\Omega}\left(\left|\frac{\partial u_1}{\partial \nu}\right|^2+
\left|\frac{\partial u_2}{\partial \nu}\right|^2
\right)(x\cdot\nu)\,dS.\endaligned
\end{equation}
where $\frac{\partial u_i}{\partial \nu}$ is the outward normal derivative exterior
 to $\partial \Omega$
at the point $x\in\partial \Omega$, $i=1,2$.
\end{lemma}

\begin{proof}
The proof is standard, see
\cite[Lemma 1.4, pp 171-172]{var}, for the reader' convenience,  we sketch the
proof here briefly. Let $u_{i}:=\frac{\partial u_i}{\partial x_{i}}$ and ${\bf n}$ be the unit outer
normal at $\partial\Omega$. Multiplying the first equation of \eqref{eq:1} (resp. the second equation
of \eqref{eq:1}) by $(x\cdot \nabla u_1)$ (resp. $(x\cdot \nabla u_2)$) and integrating on $\Omega$, we
obtain
\begin{equation}\label{51}
-\int_{\Omega}\Delta u_1(x\cdot\nabla u_1)+\omega_1\int_{\Omega} u_1(x\cdot\nabla u_1)
=\int_{\Omega}u_1\log u_1^2(x\cdot\nabla u_1 )+\frac{2p}{p+q}\int_{\Omega}|u_1|^{p-2}u_1|u_2|^{q}(x\cdot\nabla u_1)\end{equation}
and
\begin{equation}\label{52}
-\int_{\Omega}\Delta u_2(x\cdot\nabla u_2)+\omega_2\int_{\Omega} u_2(x\cdot\nabla u_2)
=\int_{\Omega}u_2\log u_2^2(x\cdot\nabla u_2 )+\frac{2q}{p+q}
\int_{\Omega}|u_1|^{p}|u_2|^{q-2}u_2(x\cdot\nabla u_2).
\end{equation}
Since
$$\aligned\hbox{div}\left((x\cdot\nabla u_i)\nabla u_i \right)&=\Delta u_i(x\cdot\nabla u_i)
+\sum_{k=1}^{N}\frac{\partial u_i}{\partial x_{k}}\frac{\partial}{\partial x_{k}}
\left(\sum_{j=1}^{N}x_{j}\frac{\partial u_i}{\partial x_{j}}\right)\\
&=\Delta u_i(x\cdot\nabla u_i)+\sum_{k=1}^{N}\left(\frac{\partial u_i}{\partial x_{k}}\right)^{2}
+\sum_{j,k=1}^{N}\frac{\partial u_i}{\partial x_{k}}x_{j}\frac{\partial^{2}u_i}{\partial x_{j}\partial x_{k}}\\
&=\Delta u_i(x\cdot\nabla u_i)+|\nabla u_i|^{2}+\frac{1}{2}\sum_{j=1}^{N}x_{j}\frac{\partial}
{\partial x_{j}} |\nabla u_i|^{2} ,\ \ i=1,2,
\endaligned$$
$$\hbox{div}\left(\frac{1}{2}|\nabla u_i|^{2}x\right)=\frac{N}{2}|\nabla u_i|^{2}+\frac{1}{2}
\sum_{j=1}^{N}x_{j}\frac{\partial}{\partial x_{j}} |\nabla u_i|^{2} ,\ \ i=1,2.$$
By the divergence theorem, we have
\begin{equation}\label{53}
\int_{\Omega}\left(\Delta u_i(x\cdot\nabla u_i)+|\nabla u_i|^{2}+\frac{1}{2}\sum_{j=1}^{N}x_{j}
\frac{\partial}{\partial x_{j}} |\nabla u_i|^{2} \right)dx=
\int_{\partial \Omega}(x\cdot\nabla u_i)\nabla u_i\cdot\nu dS,\ \ i=1,2,
\end{equation}
\begin{equation}\label{54}
\frac{N}{2}\int_{\Omega}|\nabla u_i|^2dx+\frac{1}{2}\int_{\Omega}\sum_{j=1}^{N}x_{j}
\frac{\partial}{\partial x_{j}} |\nabla u_i|^{2} dx=\frac{1}{2}
\int_{\partial \Omega}|\nabla u_i|^{2}x\cdot\nu dS,\ \ i=1,2.
\end{equation}
Using \eqref{53} and \eqref{54},  we obtain
\begin{equation}\label{55}
\int_{\Omega}\Delta u_i(x\cdot\nabla u_i)dx
=\frac{N-2}{2}\int_{\Omega}|\nabla u_i|^2dx+\int_{\partial\Omega}(x\cdot\nabla u_i)\nabla u_i
\cdot\nu dS-\frac{1}{2}\int_{\partial \Omega}u_i^2x\cdot\nu dS,\ \ i=1,2.
\end{equation}
Similarly, for $i=1,2$, direct computations yield
$$\hbox{div}(u_i^2x)= Nu_i^2+2u_i(x\cdot\nabla u_i),$$
$$\hbox{div}(u_i^2\log u_i^2x)=N u_i^2\log u_i^2+ 2u_i\log u_i^2(x\cdot\nabla u_i)
+2u_i (x\cdot\nabla u_i),$$
$$\hbox{div}(|u_1|^{p}|u_2|^{q}x)=N|u_1|^{p}|u_2|^{q}+p|u_1|^{p-2}u_1|u_2|^{q}
(x\cdot\nabla u_2)+q|u_1|^{p}|u_2|^{q-2}u_2(x\cdot\nabla u_2).$$
It follows from divergence theorem that
\begin{equation}\label{56}
N\int_{\Omega}u_i^2dx+2\int_{\Omega}
u_i(x\cdot\nabla u_i)dx=\int_{\partial \Omega}u_i^2x\cdot\nu dS,\ \ i=1,2,
\end{equation}
\begin{equation}\label{57}
N\int_{\Omega} u_i^2\log u_i^2+ 2\int_{\Omega}u_i(\log u_i^2+1)(x\cdot\nabla u_i)dx=\int_{\partial\Omega}  u_i^2\log u_i^2 x\cdot\nu dS,\ \ i=1,2,
	\end{equation}
\begin{equation}\label{58}
\aligned\ &\ \int_{\Omega}(N|u_1|^{p}|u_2|^{q}+p|u_2|^{q}|u_1|^{p-2}u_1(x\cdot\nabla u_1)+q|u_1|^{p}|u_2|^{q-2}u_2
(x\cdot\nabla u_2))dx\\
\ =&\ \int_{\partial \Omega}|u_1|^{p}|u_2|^{q}x\cdot\nu dS,\ \ i=1,2.\endaligned
\end{equation}
Plug \eqref{55}, \eqref{56}, \eqref{57} and \eqref{58} into \eqref{51} and \eqref{52}, we conclude that
$$\aligned\ &  \frac{N-2}{2} \int_{\Omega}(|\nabla u_1|^2+|\nabla u_2|^2)dx+\frac{N}{2}
\int_{\Omega}(\omega_1u_1^2+\omega_2 u_2^2)\,dx\\
\ =&\ \frac{N}{2}\int_{\Omega}(\mu_1 u_1^2(\log u_1^2-1)+\mu_2 u_2^2(\log u_2^2-1))+
\frac{2N}{p+q}\int_{\Omega}|u_1|^p|u_2|^q\,dx\\
\  &+\frac{1}{2}\int_{\partial \Omega}(2\mu_1 u_1^2+2\mu_2 u_2^2+\omega_1 u_1^2+\omega_2 u_2^2
+\mu_1 u_1^2\log u_1^2+\mu_2 u_2^2\log u_2^2)x\cdot\nu\,dS\\
\ & +\frac{2}{p+q}\int_{\partial \Omega}|u_1|^p|u_2|^q x\cdot\nu\,dS
+\int_{\partial\Omega}((x\cdot\nabla u_1)\nabla u_1\cdot\nu+(x\cdot\nabla u_2)\nabla u_2\cdot\nu)\,dS.\endaligned$$
Taking account of the fact that  $u_i=0$ and $x\cdot\nabla u_i
=x\cdot \nu\frac{\partial u_i}{\partial \nu}$ on $\partial\Omega$ $i=1,2$, we have
$$\aligned\ &\ \frac{N-2}{2} \int_{\Omega}(|\nabla u_1|^2+|\nabla u_2|^2)dx+\frac{N}{2}
\int_{\Omega}(\omega_1u_1^2+\omega_2 u_2^2)\,dx\\
\ =&\ \frac{N}{2}\int_{\Omega}(\mu_1 u_1^2(\log u_1^2-1)+\mu_2 u_2^2(\log u_2^2-1))+
\frac{2N}{p+q}\int_{\Omega}|u_1|^p|u_2|^q\,dx\\
\ & +\int_{\partial\Omega}\left(\left|\frac{\partial u_1}{\partial \nu}\right|^2+ \left|\frac{\partial u_2}{\partial \nu}\right|^2
\right)(x\cdot\nu)\,dS.\endaligned$$
\end{proof}

\begin{proof}[Proof of Theorem \ref{thm:5}]
Suppose that $u_1,\ u_2\not\equiv0$ and $(u_1,u_2)$ satisfies \eqref{eq:1}.
Multiplying the first equation of \eqref{eq:1} (resp. the second equation
of \eqref{eq:1}) by $u_1$ (resp. $u_2$) and integrating on $\Omega$, we
obtain
\begin{equation}\label{eq:61}
\int_{\Omega}(|\nabla u_1|^2+|\nabla u_2|^2+\omega_1u_1^2+\omega_2 u_2^2)=
\int_{\Omega}(\mu_1 u_1^2(\log u_1^2-1)+\mu_2 u_2^2(\log u_2^2-1))
+2\int_{\Omega}|u_1|^p|u_2|^q.
\end{equation}	
Since $\Omega$ is star-shaped with respect to $0\in\R$, we have
$x\cdot\nu>0$ for all $x\in\partial\Omega$. Thus
$\frac{\partial u_i}{\partial\nu}=0$ on $\partial\Omega$. By using Lemma \ref{le51}
and \eqref{eq:61}, we obtain
$$\aligned&\ \int_{\Omega}(|\nabla u_1|^2+|\nabla u_2|^2)
=\left(N-\frac{2N}{p+q}\right)\int_{\Omega}|u_1|^p|u_2|^q,\endaligned$$
and hence $u_1=u_2\equiv0$ since $p+q<2$.
\end{proof}

\section{Behavior of the solutions as $p+q\to 2^*$}\label{sec:segregation6}

In this section, we investigate the asymptotic behavior of the solutions for problem \eqref{eq:106}
on $\Omega=\R^N(N\geq4)$,
and by studying the behavior of system \eqref{eq:106} when $p+q$ tends to $2^*$. As a by-product, we
give a way to prove the existence of solutions to the
Br\'{e}zis-Nirenberg type problem on $\R^N$.

\begin{lemma}\label{le62}
$c_{p+q}>0$ is achieved.
\end{lemma}	

\begin{proof} We divide the proof into five steps.
	
{\bf Step 1:} For any
$(u_1,u_2)\in \mathcal{H}(\R^N)\backslash\{(0,0)\}$, there is a unique
$t_0>0$ such that $$\mathcal{I}_{p+q}(t_0 u_1,t_0 u_2)=
\max_{t>0}\mathcal{I}_{p+q}( tu_1, tu_2).$$
 Moreover, if  $\mathcal{J}_{p+q}(u_1,u_2)<0$, we obtain $ t_0<1$.

 For any
 $(u_1,u_2)\in \mathcal{H}(\R^N)\backslash\{(0,0)\}$,
 $$\aligned\mathcal{I}_{p+q}(tu_1, tu_2)&=\frac{t^2}{2}
 \int_{\R^N}(|\nabla u_1|^2+|\nabla u_2|^2+\omega_1 u_1^2+\omega_2 u_2^2)-\frac{t^2}{2}\int_{\R^N}(\mu_1u_1^2(\log t^2u_1^2-1))\\
 &\ \ \ \ -\frac{t^2}{2}\int_{\R^N}(\mu_2u_2^2(\log t^2u_2^2-1))-\frac{2}{p+q}t^{p+q}
\int_{\R^N}|u_1|^{p}|u_2|^{q}.\endaligned$$
Denote $h(t):=\mathcal{I}_{p+q}(tu_1,tu_2)$. Since $p+q>2,$ we see that
$h(t)>0$ for $t>0$ small enough and $h(t)\rightarrow-\infty$ as
$t\rightarrow+\infty,$ this implies that $h(t)$ attains its maximum.
We compute
$$\aligned h^\prime(t)&=t
\int_{\R^N}(|\nabla u_1|^2+|\nabla u_2|^2+\omega_1 u_1^2+\omega_2 u_2^2)-t(\log t^2-1)\int_{\R^N}(\mu_1u_1^2+\mu_2u_2^2)\\
&\ \ \ \ -t\int_{\R^N}(\mu_1u_1^2\log u_1^2+\mu_2u_2^2\log u_2^2)-2t^{p+q-1}
\int_{\R^N}|u_1|^{p}|u_2|^{q}.\endaligned$$
It is easy to see that  $h$ has a unique critical point $\bar{t}$.
In particular, for any $u_1,u_2\neq0$, $\bar{t}>0$ is the unique
value such that $(\bar{t}u_1,\bar{t}u_2)$ belongs to $\mathcal{N}_{p+q},$ and
$\mathcal{I}_{p+q}(\bar{t}u_1,\bar{t}u_2)$ reaches a global maximum for
$t=\bar{t}$. Moreover, if
$\mathcal{J}_{p+q}(u_1,u_2)<0$, we claim that $0<\bar{t}<1$. Indeed, if
$\mathcal{J}_{p+q}(u_1,u_2)<0$, it follows from
$$\aligned\mathcal{J}_{p+q}(u_1,u_2)&:=\int_{\R^N}\left(|\nabla u_1|^2+|\nabla u_2|^2
+\omega_1 u_1^2+\omega_2 u_2^2
-\mu_1 u_1^2\log u_1^2-2|u_1|^{p}|u_2|^{q}-\mu_2 u_2^2\log u_2^2\right)<0,\endaligned $$
and
$$\bar t^2\int_{\R^N}\left(|\nabla u_1|^2+|\nabla u_2|^2
+\omega_1 u_1^2+\omega_2 u_2^2
-\mu_1 u_1^2\log u_1^2 -\mu_2 u_2^2\log u_2^2\right)-2\bar t^{p+q}\int_{\R^N}|u_1|^{p}|u_2|^{q}=0,$$
that
$$\bar{t}^2\int_{\R^N}\left( \mu_1 u_1^2\log \bar{t}^2+\mu_2 u_2^2\log \bar{t}^2\right)
+2(\bar{t}^{p+q}-t^2)\int_{\R^N}
|u_1|^{p}|u_2|^{q}<0,$$
which implies that $0<\bar{t}<1$. This finishes the proof.

{\bf Step 2:} Let $\{(u_{1,n},u_{2,n})\}\subset\mathcal{N}_{p+q}$ be a
 minimizing sequence of $c_{p+q}$, then $\{(u_{1,n},u_{2,n})\}$ is 
 bounded in $\mathcal{H}(\R^N)$.
  Using Step 1, we know that $\mathcal{N}_{p+q}$ is
nonempty. Threfore,
$$\aligned c_{p+q}&=\lim_{n\to\infty}\mathcal{I}_{p+q}(u_{1,n},u_{2,n}) \\
&=\lim_{n\to\infty}\left(\mathcal{I}_{p+q}(u_{1,n},u_{2,n})
-\frac 12\mathcal{J}_{p+q}(u_{1,n},u_{2,n}) \right)\\
&=\lim_{n\to\infty}\left(\frac 12\int_{\R^N}(\mu_1 u_{1,n}^2+\mu_2 u_{2,n}^2)+\left(1-\frac{2}{p+q}
\right)\int_{\R^N}|u_{1,n}|^p|u_{2,n}|^q\right).\endaligned$$
Thus, $\{u_{i,n}\} $ is bounded in $L^2(\R^N)$ and in $L^{p+q}(\R^N)$. Taking
$a>0$ small enough in \eqref{eq:5}
yields
$$\int_{\R^N}u^2\log u^2\leq \frac 12|\nabla u|_2^2+C_1(\log|u|_2^2+1)|u|_2^2,\ \ \text{for all}\ \ u\in H^1(\R^N).$$
Together with $\{(u_{1,n},u_{2,n})\}\subset
\mathcal{N}_{p+q}$, H\"older inequality, we have
$$\aligned \ &\ \int_{\R^N}(|\nabla u_{1,n}|^2+\omega_1|u_{1,n}|^2+|\nabla u_{2,n}|^2+\omega_2|u_{2,n}|^2)\\
\ \leq&\  \frac 12(|\nabla u_{1,n}|_2^2+|\nabla u_{2,n}|_2^2)
+C\left(|u_{1,n}|_2^2(\log|u_{1,n}|_2^2+1)+ |u_{2,n}|_2^2(\log|u_{2,n}|_2^2+1) \right)
+2\int_{\R^N} |u_{1,n}|^p|u_{2,n}|^q\\
\ \leq &\ \frac 12(|\nabla u_{1,n}|_2^2+|\nabla u_{2,n}|_2^2)
+C\left(|u_{1,n}|_2^2(\log|u_{1,n}|_2^2+1)+ |u_{2,n}|_2^2(\log|u_{2,n}|_2^2+1) \right)+2 |u_{1,n}|_{p+q}^p|u_{2,n}|_{p+q}^q.\endaligned$$
This implies that  $\{(u_{1,n},u_{2,n})\}$ is bounded in $\mathcal{H}(\R^N)$.

{\bf Step 3:} $c_{p+q}>0$. According to Lemma \ref{lemma:1}, H\"older inequality inequality \eqref{1.12} we have
\begin{equation}\label{eq23}
\aligned \ &\ \int_{\R^N}(|\nabla u_{1,n}|^2+\omega_1|u_{1,n}|^2+|\nabla u_{2,n}|^2+\omega_2|u_{2,n}|^2)\\
\ =&\ \mu_1\left(\int_{u_{1,n}^2\geq1}u_{1,n}^2\log u_{1,n}^2
+\int_{u_{1,n}^2\leq1}u_{1,n}^2\log u_{1,n}^2\right)\\
\ &\ +\mu_2 \left(\int_{u_{2,n}^2
\geq1}u_{2,n}^2\log u_{2,n}^2
+\int_{u_{2,n}^2\leq1}u_{2,n}^2\log u_{2,n}^2\right)+ 2\int_{\R^N}|u_{1,n}|^{p}|u_{2,n}|^{q}\\
\ \le&\ \mu_1 \int_{u_{1,n}^2\geq1}u_{1,n}^2\log u_{1,n}^2
+ \mu_2 \int_{u_{2,n}^2\geq1}u_{2,n}^2\log u_{2,n}^2 +2\int_{\R^N}|u_{1,n}|^{p}|u_{2,n}|^{q}\\
\ \le&\  \frac{(N-2)}{2e}(\mu_1|u_{1,n}|_{2^*}^{{2^*}}+\mu_2|u_{2,n}|_{2^*}^{2^*})+2|u_{1,n}|_{p+q}^{p}|u_{2,n}|_{p+q}^{q}.
\endaligned
\end{equation}
From the Sobolev embedding theorem, we deduce that
\begin{equation}\label{eq24}
\int_{\R^N}(|\nabla u_{1,n}|^2+\omega_1|u_{1,n}|^2+|\nabla u_{2,n}|^2
+\omega_2|u_{2,n}|^2)\geq C_1>0.
\end{equation}
Therefore,
$$\int_{\R^N}|u_{1,n}|^{p}|u_{2,n}|^{q}\geq C_2>0. $$
Thus
$$\aligned c_{p+q}&=
\lim_{n\to\infty}\left(\frac 12\int_{\R^N}(\mu_1 u_{1,n}^2+\mu_2 u_{2,n}^2)
+\left(1-\frac{2}{p+q}\right)\int_{\R^N}|u_{1,n}|^p|u_{2,n}|^q\right)\geq C>0. \endaligned$$

{\bf Step 4:} $u_1\neq 0$ or $u_2\neq 0$.

We claim that there exist $\delta>0$ and $\{x_n\}\subset \R^N$
such that
\begin{equation}\label{eq26}
\liminf_{n\to\infty}\int_{B_1(x_n)}|u_n|^2>\delta>0,
\end{equation}
where $B_1(y):=\{z\in\R^N :|y-z|<1\}$.
Otherwise,
$$\lim_{n\to\infty}\sup_{y\in\R^N}\int_{B_1(y)}|u_n|^2=0,$$
then $u_n\to0$ in $L^s(\R^N)$ for any $s\in(2,2^*)$.
It follows by \eqref{eq23} that
$$\lim_{n\to\infty}\int_{\R^N}(|\nabla u_{1,n}|^2
+\omega_1|u_{1,n}|^2+|\nabla u_{2,n}|^2+\omega_2|u_{2,n}|^2)=0,$$
which is contradict to \eqref{eq24}. Therefore,
\eqref{eq26} holds.

Define $\nu_{1,n}=u_{1,n}(\cdot+x_{n})$, $\nu_{2,n}=u_{2,n}(\cdot+x_{n})$, where $x_{n}$ is given in
\eqref{eq26}. Since
$(u_{1,n},u_{2,n})\rightharpoonup(u_1,u_2)$ in $\mathcal{H}(\mathbb R^{N})$.
We note that $(\nu_{1,n},\nu_{2,n})$ is still a bounded minimizing sequence of
$c_{p+q}$. Then, up to a subsequence, $(\nu_{1,n},\nu_{2,n})$ is bounded in $\mathcal{H}(\mathbb R^N)$
and we may
assume that $(\nu_{1,n},\nu_{2,n})\rightharpoonup(\nu_1,\nu_2)$ in
$\mathcal{H}(\mathbb R^N)$. Moreover,
by compactness, we have that $\nu_{i,n}\to\nu_i$ in $L_{loc}^2(\mathbb R^N),\ i=1,2$.
It follows from \eqref{eq26} that $\nu_i\neq 0$.
Moreover, Lemma \ref{le332} implies that $\nu_i^2\log \nu_i^2\in L^1(\R^N),\ i=1,2.$

\vskip0.12in

{\bf Step 5:} $(\nu_1,\nu_2)\in\mathcal{N}_{p+q}$ and $\mathcal{I}_{p+q}(\nu_1,\nu_2)=c_{p+q}$.

First, assume by contradiction that
$\mathcal{J}_{p+q}(\nu_1,\nu_2)<0$. By using Step 1, we can deduce that there exists $0<t<1$
such that $\mathcal{J}_{p+q}(t\nu_1,t\nu_2)=0$. Therefore
$$\aligned c_{p+q}&\leq\mathcal{I}_{p+q}(t\nu_1,t\nu_2)\\
&=\left(\frac 12 t^2\int_{\R^N}(\mu_1\nu_1^2+\mu_2\nu_2^2)
+\left(t^{p+q}-\frac{2}{p+q}t^{p+q}\right)\int_{\R^N}|\nu_1|^p|\nu_2|^q\right)\\
&<\liminf_{n\to\infty}\left(\frac 12\int_{\R^N}(\mu_1 \nu_{1,n}^2+\mu_2 \nu_{2,n}^2)
+\left(1-\frac{2}{p+q}\right)\int_{\R^N}|\nu_{1,n}|^p|\nu_{2,n}|^q\right)\\
&=c_{p+q},\endaligned$$
which is impossible. On the other hand, assume that
$\mathcal{J}_{p+q}(\nu_1,\nu_2)>0$.
Set $\xi_{n}:=\nu_{1,n}-\nu_1,\gamma_{n}:=\nu_{2,n}-\nu_2$, by Lemma \ref{le332}
and Br\'{e}zis-Lieb lemma \ref{eq:BL}, we may obtain
\begin{equation}\label{eq3.7}
\mathcal{J}_{p+q}(\nu_{1,n},\nu_{2,n})=\mathcal{J}_{p+q}(\nu_1,\nu_2)
+\mathcal{J}_{p+q}(\xi_{n},\gamma_{n})+o_{n}(1).
\end{equation}
Then $\limsup\limits_{n\rightarrow\infty}\mathcal{J}_{p+q}(\xi_{n},\gamma_{n})<0$.
By Step 1, there exists $t_{n}\in(0,1)$ such that
$(t_{n}\xi_{n},t_{n}\gamma_{n})\in \mathcal{N}_{p+q}.$
Furthermore, one has that $\limsup\limits_{n\rightarrow\infty}t_{n}<1$,
otherwise, along a subsequence, $t_{n}\rightarrow1$ and hence
$\mathcal{J}_{p+q}(\xi_{n},\gamma_{n})
=\mathcal{J}_{p+q}(t_{n}\xi_{n},t_{n}\gamma_{n})+o_{n}(1)=o_{n}(1)$,
a contradiction.
It follows from $(\nu_{1,n},\nu_{2,n})\in\mathcal{N}_{p+q}$ and \eqref{eq3.7}
that
$$\aligned &\ c_{p+q}+o_{n}(1)\\
 =&\ \mathcal{I}_{p+q}(\nu_{1,n},\mu_{2,n})-\frac{1}{2}\mathcal{J}_{p+q}(\nu_{1,n},\mu_{2,n})\\
 =&\ \left(\frac 12\int_{\R^N}(\mu_1 \nu_{1,n}^2+\mu_2 \nu_{2,n}^2)+
 \left(1-\frac{2}{p+q}\right)\int_{\R^N}|\nu_{1,n}|^p|\nu_{2,n}|^q\right)\\
 \geq&\ \left(\frac 12\int_{\R^N}(\mu_1 (\nu_1^2+\xi_n^2)+\mu_2(\nu_{2}^2+\nu_n^2))+
 \left(1-\frac{2}{p+q}\right)\int_{\R^N}(|\nu_{1}|^p|\nu_{2}|^q+|\xi_n|^p|\nu_n|^q)\right)\\
 =&\ \mathcal{I}_{p+q}(\nu_{1},\nu_{2})-\frac{1}{2}\mathcal{J}_{p+q}(\nu_{1},\nu_{2})
 +\mathcal{I}_{p+q}(\xi_{n},\nu_{n})-\frac{1}{2}\mathcal{J}_{p+q}(\xi_{n},\nu_{n})\\
 >&\ \mathcal{I}_{p+q}(\nu_{1},\nu_{2})-\frac{1}{2}\mathcal{J}_{p+q}(\nu_{1},\nu_{2})
 +\mathcal{I}_{p+q}(t_n\xi_{n},t_n\nu_{n})-\frac{1}{2}\mathcal{J}_{p+q}(t_n\xi_{n},t_n\nu_{n})\\
 =&\ \mathcal{I}_{p+q}(t_n\xi_{n},t_n\nu_{n})+
 \left(\frac 12\int_{\R^N}(\mu_1 \nu_{1}^2+\mu_2 \nu_{2}^2)+
 \left(1-\frac{2}{p+q}\right)\int_{\R^N}|\nu_{1}|^p|\nu_{2}|^q\right) \\
\geq&\ c_{p+q}, \endaligned$$
 which is also a contradiction. Thus, we deduce that $\mathcal{J}_{p+q}(\nu_1,\nu_2)=0$, and then
 $(\nu_1,\nu_2)\in\mathcal{N}_{p+q}$. By using Fatou's Lemma and
 $(\nu_{1,n},\nu_{2,n})\in\mathcal{N}_{p+q}$, we may get
 $$\aligned c_{p+q}&\leq\mathcal{I}_{p+q}(t\nu_1,t\nu_2)\\
&=
\left(\frac 12 t^2\int_{\R^N}(\mu_1\nu_1^2+\mu_2\nu_2^2)
+\left(t^{p+q}-\frac{2}{p+q}t^{p+q}\right)\int_{\R^N}|\nu_1|^p|\nu_2|^q\right)\\
&\leq\liminf_{n\to\infty}\left(\frac 12\int_{\R^N}(\mu_1 \nu_{1,n}^2+\mu_2 \nu_{2,n}^2)
+\left(1-\frac{2}{p+q}\right)\int_{\R^N}|\nu_{1,n}|^p|\nu_{2,n}|^q\right)\\
&=c_{p+q}.\endaligned$$
Therefore, $\mathcal{I}_{p+q}(\nu_1,\nu_2)=c_{p+q}$.
 \end{proof}

\begin{lemma}\label{61}
Let $(u_{1,p+q},u_{2,p+q})$ be solutions of the system \eqref{eq:106}, that is,
\begin{equation}\label{61}
\aligned\ &\ \int_{\R^N}(\nabla u_{1,p+q}\nabla\varphi_1
+\nabla u_{2,p+q}\nabla\varphi_2+\omega_1 u_{1,p+q}\varphi_1+\omega_2 u_{2,p+q}
\varphi_2)\\
\ &\ -\int_{\R^N}\bigg(\mu_1 u_{1,p+q}\varphi_1\log u_{1,p+q}^2
+\mu_2 u_{2,p+q}\varphi_2\log u_{2,p+q}^2+\frac{2p}{p+q}|u_{2,p+q}|^{q}|u_{1,p+q}|^{p-2}u_{1,p+q}\varphi_1\\
\ &\
+\frac{2q}{p+q}|u_{1,p+q}|^{p}|u_{2,p+q}|^{q-2}u_{2,p+q}\varphi_2\bigg)=0\endaligned
\end{equation}
for all $\varphi_1,\varphi_2\in C_0^\infty(\R^N)$. For $2<p_n+q_n<2^*$ with
$p_n+q_n\to2^*$, assume that $u_{i,n}:=u_{i,p_n+q_n}\rightharpoonup u_i,\ i=1,2$
in $H^1(\R^N)$, then $u_i^2\log u_i^2\in L^1(\R^N),\ i=1,2$ and satisfies
\begin{equation}\label{62}
\aligned \ &\ \int_{\R^N}(\nabla u_1\nabla\varphi_1+
\nabla u_2\nabla\varphi_2+\omega_1\varphi_1+\omega_2\varphi_2)-\int_{\R^N}\bigg(\mu_1 u_1\varphi_1\log u_1^2+\mu_2 u_2\varphi_2\log u_2^2\\
\ &\
+\frac{2p}{2^*}|u_2|^{q}|u_1|^{p-2}u_1\varphi_1+\frac{2q}{2^*}|u_1|^{p}|u_2|^{q-2}u_2\varphi_2\bigg)=0,\quad\quad p+q=2^*\endaligned
\end{equation}
for all $\varphi_1,\varphi_2\in C_0^\infty(\R^N)$. Moreover,
$\mathcal{I}_{2^*}(u_1,u_2)\leq\liminf\limits_{n\to\infty}\mathcal{I}_{p_n+q_n}(u_{1,n},u_{2,n})$.
\end{lemma}

\begin{proof}  It is easy to see that ${(u_{1,n},u_{2,n})}$ is bounded in
$\mathcal{H}(\R^N)$ since $(u_{1,n},u_{2,n})\rightharpoonup (u_1,u_2)$ in
$\mathcal{H}(\R^N)$. Hence
\begin{equation}\label{63}
\int_{\R^N}|u_{i,n}|^2\log u_{i}^2|\leq \frac{(N-2)}{2e}|u_{i,n}|^{2^*}\leq C<+\infty,\ \ i=1,2.
\end{equation}	
On the other hand, $(u_{1,n},u_{2,n})$ satisfying
\begin{equation}\label{64}
\aligned\ &\ \int_{\R^N}(|\nabla u_{1,n}|^2+
|\nabla u_{2,n}|^2+\omega_1 u_{1,n}+\omega_2 u_{2,n}-\mu_1 u_{1,n}^2\log u_{1,n}^2-
\mu_2 u_{2,n}^2\log u_{2,n}^2-2|u_{1,n}|^{p}|u_{2,n}|^{q})=0.\endaligned
\end{equation}
Combining this with \eqref{63} implies that $\{u_{i,n}^2\log u_{i,n}^2\}$ is bounded
in $L^1(\R^N)$. Thus, we can deduce from Lemma \ref{lemma:1} that
$u_i^2\log u_i^2\in L^1(\R^N),\ i=1,2$.
By using the fact $(u_{1,n},u_{2,n})\rightharpoonup(u_1,u_2)$ in $\mathcal{H}(\R^N)$, one can verify
\eqref{61}. By Lemma \ref{le332}, Sobolev embedding and weakly lower semi-continuity we have
$$\aligned \mathcal{I}_{2^*}(u_{1},u_{2})
&= \frac 12\int_{\R^N}(\mu_1 u_{1}^2+\mu_2 u_{2}^2)+\left(1-\frac{2}{2^*}
\right)\int_{\R^N}|u_{1}|^p|u_{2}|^q \\
&= \mathcal{I}_{2^*}(u_{1},u_{2})-\frac 12\mathcal{J}_{2^*}(u_{1},u_{2}) \\
&\leq\lim_{n\to\infty}\left(\mathcal{I}_{p_n+q_n}(u_{1,n},u_{2,n})
-\frac 12\mathcal{J}_{p_n+q_n}(u_{1,n},u_{2,n}) \right)\\
&=\lim_{n\to\infty}\mathcal{I}_{p_n+q_n}(u_{1,n},u_{2,n}).\endaligned$$
Thus, the proof is completed.
\end{proof}

\begin{lemma}\label{63}
Let $(u_{1,p+q},u_{2,p+q})$ and $(u_{1,n},u_{2,n})$ be as above. Suppose that
$$\lim\limits_{n\to\infty}\mathcal{I}_{p_n+q_n}(u_{1,n},u_{2,n})\in\left(0,\frac 1N S^{\frac N2}\right),$$
where $S$ is the best
constant of
Sobolev imbedding from $H^1(\R^N)$ into $L^{2^*}(\R^N)$. Then, up to
translations,
$u_{i,n}\rightharpoonup u_i\neq0$ in $H^1(\R^N)$.
\end{lemma}

\begin{proof} Since $\{\mathcal{I}_{p_n+q_n}(u_{1,n},u_{2,n})\}$ is uniformly bounded and
	 $(u_{1,n},u_{2,n})$ are
	 solutions for system
	 \eqref{eq:106} with $p+q=p_n+q_n$.
Taking $a>0$ small enough in \eqref{eq:5} yields
$$\int_{\R^N} u_i^2\log u_i^2\leq\frac 12|\nabla u_i|_2^2+C_1(\log|u_i|_2^2+1)|u_i|_2^2,\ \
\text{for all}\ \ i=1,2 \ \ \text{and}\ \ u_i\in H^1(\R^N).$$
Since $ \mathcal{J}_{p_n+q_n}(u_{1,n},u_{2,n})=0$, we have that
\begin{equation}\label{65}
\aligned \ &\ \int_{\R^N}(|\nabla u_{1,n}|^2+\nabla u_{2,n}|^2+
\omega_1 u_{1,n}^2
+\omega_2 u_{2,n}^2 )\\
\ \leq&\ \frac{\mu_1}{2}|\nabla u_1|_2^2+C_1\mu_1(\log |u_1|_2^2+1)|u_1|_2^2
+\frac{\mu_2}{2}|\nabla u_2|_2^2+C_1\mu_2(\log |u_2|_2^2+1)|u_2|_2^2
+2\int_{\R^N}|u_{1,n}|^p|u_{2,n}|^q,\endaligned
\end{equation}
which implies that $\{(u_{1,n},u_{2,n})\}$ is bounded in $\mathcal{H}(\R^N)$.
We claim that $u_1,\ u_2\neq 0$.
If $|u_i|_2=0,\ i=1,2$, according to interpolation inequality, we have $|u_{i,n}|_r\to0$
for any $2<r<2^*,\ i=1,2$.
By using the fact $ \mathcal{J}_{p_n+q_n}(u_{1,n},u_{2,n})=0$, we obtain
$$\aligned\ &\int_{\R^N}(|\nabla u_{1,n}|^2+|\nabla u_{2,n}|^2+\omega u_{1,n}^2
+\omega u_{2,n}^2 )\\
\  =&\ \mu_1\int_{\R^N} u_{1,n}^2\log u_{1,n}^2
+\mu_2\int_{\R^N} u_{2,n}^2\log u_{2,n}^2
+2\int_{\R^N}|u_{1,n}|^{p_n}|u_{2,n}|^{q_n}\\
\ \leq&\ \frac{2p}{p+q}\int_{\R^N}|u_{1,n}|^{p_n+q_n}+\frac{2q}{p+q}\int_{\R^N}|u_{2,n}|^{p_n+q_n}\\
\ \leq&\ C_{p_0,q_0}\int_{\R^N}(|u_{1,n}|^{p_0+q_0}+|u_{2,n}|^{p_0+q_0})
+\frac{p_n+q_n-(p_0+q_0)}{2^*-(p_0+q_0)}
\int_{\R^N}(|u_{1,n}|^{2^*}+|u_{2,n}|^{2^*})\\
\  &\ +\frac{2^*-(p_n+q_n)}{2^*-(p_0+q_0)}
\int_{\R^N}(|u_{1,n}|^{p_0+q_0}+|u_{2,n}|^{p_0+q_0})\\
\ =&\ \int_{\R^N}|u_{1,n}|^{2^*}+\int_{\R^N}|u_{2,n}|^{2^*}+o_n(1)\\
\ \leq&\ S^{\frac{N}{2-N}}\left(\left(\int_{\R^N}|\nabla u_{1,n}|^2\right)
^{\frac{N}{N-2}}+
\left(\int_{\R^N}|\nabla u_{2,n}|^2\right)^{\frac{N}{N-2}}\right)+o_n(1),\endaligned$$
where $2<p_0+q_0<p_n+q_n<2^*-\delta$ for some positive constant $\delta$.
From this we deduce that either
$$\int_{\R^N}(|\nabla u_{1,n}|^2+|\nabla u_{2,n}|^2)\to0\ \ \text{or}\ \
\liminf\limits_{n\to\infty}\int_{\R^N}(|\nabla u_{1,n}|^2+|\nabla u_{2,n}|^2)
\geq S^{\frac{N}{2}}.$$
 If $$\int_{\R^N}(|\nabla u_{1,n}|^2+|\nabla u_{2,n}|^2)\to0,$$ then
 $$\int_{\R^N} (|u_{1,n}|^{2^*}+|u_{2,n}|^{2^*})\to0,$$ $u_{i,n}\to0$
 in $H^1(\R^N),\ i=1,2$
 and $\mathcal{I}_{p_n+q_n}(u_{1,n},u_{2,n})\to0$, a contradiction.
 If $$\liminf\limits_{n\to\infty}\int_{\R^N}(|\nabla u_{1,n}|^2+|\nabla u_{2,n}|^2)
 \geq S^{\frac{N}{2}} ,$$
 then $$\liminf\limits_{n\to\infty}2\int_{\R^N}|u_{1,n}|^{p_n}|u_{2,n}|^{q_n} \geq S^{\frac{N}{2}} ,$$
thus $$\liminf\limits_{n\to\infty}\mathcal{I}_{p_n+q_n}(u_{1,n},u_{2,n})
=\liminf\limits_{n\to\infty}\left(\frac 12\int_{\R^N}(|u_{1,n}|^2+|u_{2,n}|^2)
+\frac{p+q-2}{p+q}\int_{\R^N}|u_{1,n}|^{p}|u_{2,n}|^{q}\right)\geq \frac 1N  S^{\frac{N}{2}},$$
 which is also a contradiction. So, for every $\rho_1,\rho_2>0$, we get $u_i\neq 0,\ i=1,2$.
\end{proof}

\begin{lemma}\label{64}
If $N\geq 4$, then $c_{2^*}<\frac {1}{N}S^{\frac N2}$.
\end{lemma}

\begin{proof} We follow the   idea of Br\'{e}zis-Nirenberg \cite{5} and also  Alves, de Morais Filho and Souto
\cite{alves}. We look for a nonnegative function $(u_1,u_2)$ such that
$\sup\limits_{t\geq 0}\mathcal{I}_{2^*}(tu_1,tu_2)<\frac 1N S^{\frac N2}$. Suppose
$\mathcal{I}_{2^*}(t_0u_1,t_0u_2)=\sup\limits_{t\geq0}(tu_1,tu_2)$,
then $\mathcal{J}_{2^*}(t_0u_1,t_0u_2)=0$ and
$\mathcal{I}_{2^*}(t_0u_1,t_0u_2)<\frac 1N S^{\frac N2}$.
Hence $c_{2^*}<\mathcal{I}_{2^*}(t_0u_1,t_0u_2)< \frac 1N S^{\frac N2}$.
One of the possible candidates for $u$ is $u_{\varepsilon}=\phi w_\varepsilon$, where
$\phi $ is a smooth cut-off function such that $\phi(x)=1$ if $|x|\leq1 \
\phi(x)=0$ if $|x|\geq2$ and $|\nabla\phi|\leq2$; and
$$w_\varepsilon(x)=\frac{(N(N-2)\varepsilon)^{\frac{N-2}{4}}}{(\varepsilon+|x|^2)^{\frac{N-2}{2}}} .$$
Note that the function $w_\varepsilon$  solves the equation
$\Delta w_\varepsilon+w_\varepsilon^{\frac{N+2}{N-2}} =0$. Following \cite{5}, for
sufficiently small $\varepsilon$, by a direct computation we estimate the terms of
$ \mathcal{I}_{2^*}(tu_{1,\varepsilon},tu_{2,\varepsilon})$ as
follows
 $$ \int_{\R^N}|\nabla u_\varepsilon|^2=S^{\frac N2}+O(\varepsilon^{\frac{N-2}{2}}),\ \
 \int_{\R^N} |u_\varepsilon|^{2^*}=S^{\frac N2}+O(\varepsilon^{\frac N2}),$$
 $$\int_{\R^N} |u_\varepsilon|^2\log u_\varepsilon ^2 \geq
 \left\{\aligned\ &d_1\varepsilon |\ln \varepsilon|,\ \ &if\ \ N\geq5,\\
\ &d_2 \varepsilon |\ln \varepsilon |^2,\ \ &if\ \  N=4,\endaligned\right.$$
where $\varepsilon$ is a positive small constant. According to \cite[Theorem 5]{alves},
\begin{equation}\label{eq6.7}
	 \tilde{S}=\inf_{(u_1,u_2)\in \mathcal{H}(\mathbb R^N)\backslash \{{(0,0)}\}}=
\frac{\int_{\R^N}(|\nabla u_1|^2+|\nabla u_2|^2)}
{\left(\int_{\R^N}|u_1|^p|u_2|^q\right)^{\frac{2}{p+q}}}
=\left(\left(\frac pq\right)^{\frac{q}{p+q}}+\left(\frac qp\right)^{\frac{p}{p+q}}\right)S,
\ \ p+q\leq 2^*.
\end{equation}
Similarly, we can compute
 $$ \int_{\R^N}(|\nabla u_{1,\varepsilon}|^2
+|\nabla u_{2,\varepsilon}|^2) =\tilde{S}^{\frac N2}+O(\varepsilon^{\frac{N-2}{2}}),\ \
\int_{\R^N}|u_{1,\varepsilon}|^p|u_{2,\varepsilon}|^q=\tilde{S}^{\frac N2}+O(\varepsilon^{\frac N2}).$$
It is easy to see that there exist
$\varepsilon_0>0$ and $0<T_1<T_2$ such that for $\varepsilon\leq\varepsilon_0$,
the function $t\mapsto \mathcal{I}_{2^*}(tu_{1,\varepsilon},tu_{2,\varepsilon})$ achieves its maximum at
some $t_0\in[T_1,T_2]$. Hence, if $N\geq4$, by using \eqref{eq6.7} we have
$$\aligned \sup_{t\geq0}\mathcal{I}_{2^*}(tu_{1,\varepsilon},tu_{2,\varepsilon})&=
\frac{t_0^2}{2}\int_{\R^N}(|\nabla u_{1,\varepsilon}|^2
+|\nabla u_{2,\varepsilon}|^2+\omega_1 u_{1,\varepsilon}^2
+\omega_2 u_{2,\varepsilon}^2)\\
&\ \ \ \ -\frac{t_0^2}{2}\int_{\R^N}
\left(\mu_1 u_{1,\varepsilon}^2(\log u_{1,\varepsilon}^2-1)+
\mu_2 u_{2,\varepsilon}^2(\log u_{2,\varepsilon}^2-1)\right)
-\frac{2t_0^{2^*}}{2^*}\int_{\R^N}|u_{1,\varepsilon}|^{p}|u_{2,\varepsilon}|^{q}\\
&\leq \left(\frac{t_0^2}{2}-\frac{2t_0^{2^*}}{2^*}\right) \tilde{S}^{\frac N2}-
\varepsilon |\ln \varepsilon|+O\left( \varepsilon^{\frac{(N-2)(2-\delta)}{4}}\right)\\
&\leq  \frac 1N S^{\frac N2}\endaligned$$
for $\varepsilon$ small enough. Thus we complete the proof.
\end{proof}

\begin{proof}[Proof of Theorem \ref{thm:6}]
By Lemma \ref{64}, $c_{2^*}<\frac 1N S^{\frac N2}$ if $ N\geq 4$. We
claim that
the minimum of $\mathcal{I}_{p+q}|_{\mathcal{N}_{p+q}}$ is indeed a solution of
\eqref{eq:1071} (see \cite[Lemma 2.5]{lww},\cite[Theorem 1.1]{mma22} and
\cite[Proof of Theorem 1.1]{chenzhang}).  Let $(\tilde{u}_1,\tilde{u}_2)\in
\mathcal{N}_{p+q}$  be a minimizer of the
functional $\mathcal{I}_{p+q}|_{\mathcal{N}_{p+q}}$. Then from Lemma \ref{le62} , one has
$$\mathcal{I}_{p+q}(\tilde{u}_1,\tilde{u}_2)=\inf_{(u_1,u_2)\in \mathcal{H}(\R^N)}
\max\limits_{t>0}\mathcal{I}_{p+q}(tu_{1},tu_{2})=c_{p+q}.$$
Suppose by contradiction that $(\tilde{u}_1,\tilde{u}_2)$ is not
a weak solution of \eqref{eq:1071}. Then, one can find $\phi_{1},\phi_{2}
\in C_{0}^{\infty}(\mathbb R^{N})$  such that
$$\aligned\ &\ \langle \mathcal{I}_{p+q}'(\tilde{u}_1,\tilde{u}_2),(\phi_{1},\phi_{2})\rangle\\
\ =&\ \int_{\mathbb R^{N}}(\nabla\tilde{u}_1\nabla\phi_{1}+\omega_1
\tilde{u}_1\phi_{1}+
\nabla \tilde{u}_2\nabla\phi_{2}+\omega_2\tilde{u}_2\phi_{2})\\
\ &\ -\int_{\mathbb R^{N}}\bigg(\mu_1 \tilde{u}_1\phi_{1}\log\tilde{u}_1^2
+\mu_2 \tilde{u}_2\phi_{2}
\log \tilde{u}_2^2+\frac{2p}{p+q}|\tilde{u}_2|^{q}|\tilde{u}_1|^{p-2}\tilde{u}_1\phi_{1} +\frac{2q}{p+q}|\tilde{u}_1|^{p}|\tilde{u}_2|^{q-2}\tilde{u}_2\phi_{2}\bigg)\\
\ <&\ -1.\endaligned$$
We choose small  $\varepsilon> 0$ such that
$$\langle \mathcal{I}_{p+q}' (t\tilde{u}_{1}+\sigma\phi_{1},t\tilde{u}_{2}+\sigma\phi_{2}),
(\phi_{1},\phi_{2}) \rangle\leq -\frac{1}{2},\ \ |t-1|,|\sigma|\leq\varepsilon, $$
and introduce a cut-off function $0\leq \zeta\leq1 $ satisfying $\zeta(t)=1 $
for $|t-1|\leq\frac{\varepsilon}{2}$ and $\zeta(t)=0 $ for $|t-1|\geq\varepsilon. $
For $ t\geq0$, we define
$$\gamma_{1}(t):=\left\{\aligned &t\tilde{u}_{1}, &\hbox{if}\ |t-1|\geq\varepsilon,\\
&t\tilde{u}_{1}+\varepsilon\zeta(t)\phi_{1}, &\hbox{if}\ |t-1|<\varepsilon,\endaligned\right. $$
$$\gamma_{2}(t):=\left\{\aligned &t\tilde{u}_{2}, &\hbox{if}\ |t-1|\geq\varepsilon,\\
&t\tilde{u}_{2}+\varepsilon\zeta(t)\phi_{2}, &\hbox{if}\ |t-1|<\varepsilon.\endaligned\right. $$
Then $ \gamma_{i}(t)$ is a continuous curve and, eventually choosing a
smaller $ \varepsilon$, we get that $\|(\gamma_{i}(t),0)\|_{H^1(\R^N)}>0,$ for $|t-1|<\varepsilon$, $i=1,2$.

Next we claim that $\sup\limits_{t\geq0}\mathcal{I}_{p+q}(\gamma_{1}(t),\gamma_{2}(t))<c_{p+q}.$
Indeed, if $|t-1|\geq\varepsilon $, then $$\mathcal{I}_{p+q}(\gamma_{1}(t),\gamma_{2}(t))
=\mathcal{I}_{p+q}(t\tilde{u}_{1},t\tilde{u}_{2})<\mathcal{I}_{p+q}(\tilde{u}_1,\tilde{u}_2)=c_{p+q}.$$
If $ |t-1|<\varepsilon$, by using the mean value theorem to the $C^{1}$  map
$[0,\varepsilon]\ni \sigma\mapsto \mathcal{I}_{p+q}(t\tilde{u}_{1}+\sigma\zeta(t)
\phi_{1},t\tilde{u}_{2}+\sigma\zeta(t)\phi_{2})\in\mathbb R$, we find, for a
suitable $\bar{\sigma}\in(0,\varepsilon)$,
$$\aligned\ &\ \mathcal{I}_{p+q}(t\tilde{u}_{1}+\sigma\zeta(t)\phi_{1},
t\tilde{u}_{2}+\sigma\zeta(t)\phi_{2})\\
\ =&\ \mathcal{I}_{p+q}(t\tilde{u}_{1},t\tilde{u}_{2})
+\langle \mathcal{I}_{p+q}'(t\tilde{u}_{1}+\bar{\sigma}\zeta(t)\phi_{1},t\tilde{u}_{2}
+\bar{\sigma}\zeta(t)\phi_{2}),(\zeta(t)\phi_{1},\zeta(t)\phi_{2})\rangle\\
\ \leq&\ \mathcal{I}_{p+q}(t\tilde{u}_{1},t\tilde{u}_{2})-\frac{1}{2}\zeta(t)\\
\ <&\ c_{p+q}.\endaligned$$
To conclude observe that
$$G(\gamma_{1}(1-\varepsilon),\gamma_{2}(1-\varepsilon))>0,\ \
G(\gamma_{1}(1+\varepsilon),\gamma_{2}(1+\varepsilon))<0.$$
By using the continuity of the map $t\mapsto G(\gamma_{1}(t),\gamma_{2}(t))$
there exists $t_{0}\in(1-\varepsilon,1+\varepsilon) $
such that $G(\gamma_{1}(t_{0}),\gamma_{2}(t_{0}))=0$. Namely,
$$ (\gamma_{1}(t_{0}),\gamma_{2}(t_{0}))=
(t_{0}\tilde{u}_{1}
+\varepsilon\zeta(t_{0})\phi_{1},t_{0}\tilde{u}_{2}+\varepsilon\zeta(t_{0})\phi_{2})\in
\mathcal{N}_{p+q},$$  and $\mathcal{I}_{p+q}(\gamma_{1}(t_{0}),\gamma_{2}(t_{0}))<c_{p+q}$,
this is a contradiction.
According to  Lemma \ref{le62},
for $2<p+q<2^*$, system \eqref{eq:1071}
has a non-negative solution $(u_{1,p+q},u_{2,p+q})$ with
$\mathcal{I}_{p+q}(u_{1,p+q},u_{2,p+q})=c_{p+q}$. We claim that
\begin{equation}\label{eq6.9}
	\limsup\limits_{p+q\to 2^*}\mathcal{I}_{p+q}(u_{1,p+q},u_{2,p+q})\leq c_{2^*}.
\end{equation}	
Indeed, given $0<\varepsilon <\frac 12$, we find $(u_1,u_2)\in\mathcal{N}_{p+q}$
such that
$\mathcal{I}_{2^*}(u_1,u_2)\leq c_{2^*}+\varepsilon $. We first
choose $T>0$ such that $\mathcal{J}_{2^*}(Tu_1^\pm,Tu_2^\pm)\leq 1$,
where $u_i^+=\max\{u_i,0\},\ u_i^-=\min\{u_i,0\},\ i=1,2$. We then
choose $\delta>0$ such that
\begin{equation}\label{eq6.8}
|\mathcal{I}_{p+q}(tu_1^\pm,tu_2^\pm)-\mathcal{I}_{2^*}(tu_1^\pm,tu_2^\pm)|+
|\mathcal{J}_{p+q}(tu_1^\pm,tu_2^\pm)-\mathcal{J}_{2^*}(tu_1^\pm,tu_2^\pm)|
<\varepsilon
\end{equation}
for $2^*-\delta\leq p+q\leq 2^*$ and $0\leq t\leq T$.
Since $\mathcal{J}_{p+q}(0,0)=0$, $\mathcal{J}_{p+q}(Tu_1^\pm,Tu_2^\pm)\leq-\frac12$
for all $2^*-\delta\leq p+q\leq 2^*$, there exist
$t^+,\ t^-\in (0,T)$ such that
$(\bar u_1,\bar u_2):=(t^+u_1^++ t^-u_1^-,t^+u_2^++t^-u_2^-)\in\mathcal{N}_{p+q}.$
Hence, according to \eqref{eq6.8}, we get
$$\aligned c_{p+q}&\leq \mathcal{I}_{p+q}(\bar u_1,\bar u_2)\\
&\leq
\mathcal{I}_{p+q}(t^+u_1^+,t^+u_2^+)
+\mathcal{I}_{p+q}(t^-u_1^-,t^-u_2^-)\\
&\leq
\mathcal{I}_{2^*}(t^+u_1^+,t^+u_2^+)
+\mathcal{I}_{2^*}(t^-u_1^-,t^-u_2^-)+ 2\varepsilon\\
&\leq\mathcal{I}_{2^*}(u_1,u_2)+2\varepsilon \\
&\leq c_{2^*}+3\varepsilon,\endaligned$$
for $2^*-\delta\leq p+q\leq 2^*$.
This completes the proof of the claim since
$\varepsilon$ is arbitrarily. By Lemma \ref{63}, for a
subsequence $p_n+q_n\to 2^*$ up to translations,we have
$u_{i,n}:=u_{i,p_n+q_n}\rightharpoonup
u_i\neq 0$ in $\mathcal{H}(\R^N)$. By \eqref{eq6.9}, Lemmas \ref{32} and \ref{61}, we see that
$(u_1,u_2)\in \mathcal{H}(\R^N)$ solves problem \eqref{eq:1071} and satisfies
$$\mathcal{I}_{2^*}(u_{1},u_{2})\leq\limsup\limits_{p+q\to 2^*}\mathcal{I}_{p+q}(u_{1,p+q},u_{2,p+q})
\leq c_{2^*}.$$
But $c_{2^*}\leq \mathcal{I}_{2^*}(u_{1},u_{2})$ for any nontrivial
critical point $(u_{1},u_{2})$ of $\mathcal{I}_{2^*}$ since
$$\mathcal{J}_{2^*}(u_{1},u_{2})=\langle\mathcal{I}_{2^*}^\prime(u_{1},u_{2}),(u_{1},u_{2})\rangle=0.$$
By the proof of Step 4
in Lemma \ref{le62} , we know that $u_1\neq 0$ or $u_2\neq0$, without loss of
generality, we assume $u_1\neq 0$.
According to \cite{log}, we take
$$\beta(s)=\mu_2 s\log s^2+\frac{2q}{p+q}u_1^p s^{q-1}-\omega_2 s,$$
since $\beta$ is continuous,
nondecreasing for $s$ small, $\beta(0)=0$ and
$$\beta\left(\sqrt{e^{\frac{\omega_2}{\mu_2}-\frac{2q}{(p+q)\mu_2 u_1^p s^{q-2}}}}\right)=0,$$
by \cite[Theorem 1]{28} we have that each solution $u_2\geq0$
of \eqref{eq:1}
such that $u_2\in L^1_{loc}(\R^N)$ and $\Delta u_2\in L^1_{loc}(\R^N)$
in the sense of distribution, is either trivial or stricly positive. If $u_2\equiv0$,
similar to the above process, we can get $u_1>0$.
\end{proof}

\vskip0.20in
\noindent{\bf Acknowledgments}

\vskip0.20in
The authors are very grateful to the anonymous referees for their careful reading and for their very helpful comments and suggestions.

\end{document}